\documentclass[leqno,11pt]{article}
\usepackage{amssymb,amsmath}

\usepackage[all]{xy}

\newenvironment{proof}{{\bf Proof. }}{\par}{\bigskip}
\newtheorem{theo}{Theorem}[section]
\newtheorem{defi}[theo]{Definition}

\newtheorem{lem}[theo]{Lemma}
\newtheorem{prop}[theo]{Proposition}
\newtheorem{rem}[theo]{Remark}
\newtheorem{coro}[theo]{Corollary}
\newtheorem{exam}[theo]{Example}

\newcommand{\ggot}{\ensuremath{\mathfrak{g}}}
\newcommand{\hgot}{\ensuremath{\mathfrak{h}}}
\newcommand{\kgot}{\ensuremath{\mathfrak{k}}}

\newcommand{\ngot}{\ensuremath{\mathfrak{n}}}
\newcommand{\glgot}{\ensuremath{\mathfrak{gl}}}

\newcommand{\pgot}{\ensuremath{\mathfrak{p}}}
\newcommand{\qgot}{\ensuremath{\mathfrak{q}}}

\newcommand{\slgot}{\ensuremath{\mathfrak{sl}}}
\newcommand{\tgot}{\ensuremath{\mathfrak{t}}}

\newcommand{\Rgot}{\ensuremath{\mathfrak{R}}}

\newcommand{\Acal}{\ensuremath{\mathcal{A}}}

\newcommand{\Ccal}{\ensuremath{\mathcal{C}}}
\newcommand{\Dcal}{\ensuremath{\mathcal{D}}}
\newcommand{\Ecal}{\ensuremath{\mathcal{E}}}
\newcommand{\Fcal}{\ensuremath{\mathcal{F}}}

\newcommand{\Lcal}{\ensuremath{\mathcal{L}}}

\newcommand{\Ocal}{\ensuremath{\mathcal{O}}}

\newcommand{\Ucal}{\ensuremath{\mathcal{U}}}

\newcommand{\Wcal}{\ensuremath{\mathcal{W}}}

\newcommand{\Z}{\ensuremath{\mathbb{Z}}}
\newcommand{\C}{\ensuremath{\mathbb{C}}}

\newcommand{\R}{\ensuremath{\mathbb{R}}}
\newcommand{\N}{\ensuremath{\mathbb{N}}}

\newcommand{\Lbb}{\ensuremath{\mathbb{L}}}

\newcommand{\mm}{\ensuremath{\hbox{\rm \bf m}}}

\newcommand{\croc}{\ensuremath{\hookrightarrow}}

\newcommand{\indice}{\ensuremath{\hbox{\rm Index}}}
\newcommand{\Rfor}{\ensuremath{\hat R}}

\newcommand{\Char}{\ensuremath{\hbox{\rm Char}}}

\newcommand{\hatom}{\ensuremath{\hbox{\rm Hom}}}

\newcommand{\RR}{\ensuremath{\mathrm{RR}}}
\newcommand{\Thom}{\operatorname{Thom}}

\newcommand{\Ko}{\ensuremath{{\mathbf K}^0}}

\newcommand{\Sym}{\ensuremath{{\rm Sym}}}

\def \K {{\rm \bf K}}
\def \T {{\rm T}}
\def \U {{\rm U}}

\def \wtde {\widetilde}

\def \hom {{\rm hom}}

\def \semi {{\rm ss}}

\def \clif {{\bf c}}

\usepackage[usenames,dvipsnames,svgnames,table]{xcolor}

\begin{document}

\title{Stability property of multiplicities of group representations}

\author{Paul-Emile PARADAN\footnote{Institut Montpelli\'{e}rain Alexander Grothendieck, CNRS UMR 5149, Universit\'{e} de Montpellier, \texttt{paradan@math.univ-montp2.fr}}}

\maketitle

{\small
\tableofcontents}

\section{Introduction}

Recently, Stembridge \cite{Stembridge14} has proposed to generalize a classical result of Murnaghan \cite{Murnaghan38} 
by introducing the notion of {\bf stability}, a notion that can be formalized as follows. 

Let $\mm : \Ccal\to \N$ be a map defined on a semi-group $\Ccal$. An element $x\in \Ccal$ is called {\bf stable} if $\mm(x)>0$ 
and if the sequence $\mm(y+nx)$ converge for any $y\in \Ccal$, and is called {\bf semi-stable} if $\mm(nx)=1$ for all $n\in\N$. 
In this paper we will study weak version of stability: $x\in \Ccal$ is  called {\bf weakly stable} if the sequence 
$\mm(y+nx)$ is bounded for any $y\in \Ccal$, and is called {\bf weakly semi-stable} if the sequence $\mm(nx)$ is bounded. 

In the case were $\mm$ is the map defined by the Kronecker coefficients, Stembridge \cite{Stembridge14} has shown that stable 
points are semi-stable, and the converse statement was proved by Sam and Snowden \cite{Sam-Snowden}.

The main purpose of this paper is to study these stability properties in a more general setting. 
Let $\rho: G\to \tilde{G}$ be a morphism between compact connected Lie groups. We can associate to it a map 
$\mm_\rho: \Ccal_{G}\times \Ccal_{\tilde{G}}\to \N$ where $\Ccal_G$, and $\Ccal_{\tilde{G}}$ are the semi-groups 
of dominant weights that parametrized the irreducible representations of $G$ and $\tilde{G}$, and 
$\mm_\rho(\mu,\tilde{\mu})$ is defined as the multiplicity of the representation $V_{\mu}^G$ in the restriction $V_{\tilde{\mu}}^{\tilde{G}}\vert_{G}$. In this context we generalize the results of Stembridge, Sam and Snowden by proving 
that ``stability''$=$``semi-stability'' and ``weak stability''$=$``weak semi-stability''.

When $x\in\Ccal$ is semi-stable we can define a map $\mm^x: \Ccal\to \N$ : the value $\mm^x(y)$ is the limit of 
the sequence $\mm(y+nx)$ when $n\to\infty$. We will give a formula for these stretched coefficients when we work 
with the map $\mm_\rho$. It generalizes some computations done by Brion \cite{Brion93}, Manivel \cite{Manivel97} and 
Montagard \cite{Montagard96} in the plethysm case. 
In fact we are able to give a formula of $\mm(y+nx)$ when $x$ is weakly stable and $n$ is large enough.

Another interesting question is to produce examples of stable elements. In the case of Kronecker coefficients, 
Vallejo \cite{Vallejo14} and Manivel \cite{Manivel14} introduced a notion of ``additive matrix'' that permits 
them to parametrize a large family of stable elements. In Section 5 we show that this notion can be transferred 
to the morphism case $\rho$ (see Definition \ref{def:lambda-adapted}), and we  
compute the stretched coefficients associated to it.

\medskip

\medskip

The method used in this paper is explained in the next section. The overall strategy is to obtain these 
stability properties and the computation of the stretched coefficients as an application of the credo 
``[Q,R]=0" of Guillemin-Sternberg \cite{Guillemin-Sternberg82}.

\medskip

\medskip

{\bf Acknowledgements}.  The author would like to thank Pierre-Louis Montagard and Boris Pasquier 
for many valuable discussions, and Mich\`{e}le Vergne for her insightful comments.

\section{Statement of the results}

Let $M$ be a compact complex manifold acted on by a compact Lie group $G$. Let $\Lcal\to M$ be a 
$G$-equivariant holomorphic line bundle that is assumed to be {\bf ample}. Note that the $G$-action on 
$\Lcal\to M$ extends to the complex reductive group $G_\C$ \cite{Guillemin-Sternberg82}.

In this context, we are interested in the family of $G$-modules $\Gamma(M,\Lcal^{\otimes n})$ formed by the holomorphic sections, and more particularly to the sequence $\mathbf{H}(n):=\dim \Gamma(M,\Lcal^{\otimes n})^G,  n\geq 1$.
For any holomorphic $G$-complex vector bundle $\Ecal\to M$, we consider also the sequence 
$$
\mathbf{H}_\Ecal(n):=\dim \Gamma(M,\Ecal\otimes\Lcal^{\otimes n})^G, \ n\geq 1.
$$

Our main result, that we will detail in the next Section, can be summarized as follows : 
if the sequence $\mathbf{H}(n)$ is bounded, then the sequence $\mathbf{H}_\Ecal(n)$ is bounded 
for any holomorphic $G$-complex vector bundle $\Ecal$ and we can 
compute its value for large $n$.

\subsection{Stability result}

Since the line bundle $\Lcal$ is ample, there exists an Hermitian metric $h$ on $\Lcal$ such that 
the curvature $\Omega:=i(\nabla^h)^2$ of its Chern connection $\nabla^h$ is a K\"{a}hler class :  $\Omega$ is a symplectic form on $M$ that is compatible with the complex structure. By an averaging process we can assume that the $G$-action 
leaves the metric and connection invariant. Hence we have a moment map $\Phi:M\to\ggot^*$ 
defined by Kostant's relations 
\begin{equation}\label{eq:kostant-rel}
    L(X)-\iota(X_M)\nabla^h =i\langle\Phi,X\rangle\quad \mathrm{for\ all}\quad X\in\ggot.
\end{equation}
Here $L(X)$ is the Lie derivative on the sections of $\Lcal$, and $X_M(m):=\frac{d}{ds}e^{-s X}\cdot m\vert_{s=0}$ is the vector field generated by $X\in\ggot$.

The $[Q,R]=0$ Theorem of Meinrenken \cite{Meinrenken98} and Meinrenken-Sjamaar \cite{Meinrenken-Sjamaar} 
says that the moment map $\Phi$ gives a 
geometric interpretation of the sequence $\mathbf{H}(n)$.  An important object here is the reduced space
$$
M_0:=\Phi^{-1}(0)/G
$$
which is homeomorphic to the Mumford GIT quotient $M/\!\!/ G_\C$ \cite{Kirwan84}.

A special case of the $[Q,R]=0$ Theorem is the following basic but important fact that is explained in Section \ref{sec:Kahler}.

\begin{prop}\label{prop:QR-sjamaar}
We have the following equivalences:

$\bullet$ $\mathbf{H}(n)=0,\ \forall n\geq 1\ \Longleftrightarrow \ M_0=\emptyset$,

$\bullet$ $\mathbf{H}(n)$ is non-zero and bounded $\Longleftrightarrow \ M_0=\{pt\}$.

When $M_0=\{pt\}$, we have $\mathbf{H}(n):=\dim [\Lcal^{\otimes n}\vert_{m_o}]^H$ 
where $m_o\in \Phi^{-1}(0)$ and $H$ is the stabilizer subgroup of $m_o$. In particular if 
$\mathbf{H}(1)\neq 0$, then $\mathbf{H}(n)=1$ for all $n\geq 1$.
\end{prop}

Let us recall the geometric criterion that characterizes the fact that the reduced space $M_0$ 
is a singleton. We consider the tangent space $\T_{m_o}M$ attached to $m_o\in \Phi^{-1}(0)$: it is a complex 
$H$-module where $H$ is the stabilizer subgroup of $m_o$ acts. 
We consider the complex subspace $\ggot_\C\cdot m_o\subset \T_{m_o}M$ which is the tangent space at $m_o$ 
of the complex orbit $G_\C \cdot m_o$.

The following $H$-module is important for our purpose:
\begin{equation}\label{eq:symplectic-orthogonal}
\mathbb{W}:=\T_{m_o}M/\ggot_\C\cdot m_o.
\end{equation}

Let us denote $\Sym(\mathbb{W}^*)$ the $H$-module formed by the polynomial functions on $\mathbb{W}$. The following standard fact is explained in Section \ref{sec:Kahler}.

\begin{prop}\label{prop:equivalences}
We have  $\Phi^{-1}(0)=G m_o$ if and only if the $H$-multiplicities of $\Sym(\mathbb{W}^*)$ are finite.
\end{prop}

\medskip

Our main contribution is the following stability result.

\medskip

\noindent {\bf Theorem A} \ {\em Let $\Ecal\to M$ be an holomorphic $G$-vector vector bundle.
\begin{itemize}
\item If  $\mathbf{H}(n)=0,\, \forall n\geq 1$, then $\mathbf{H}_\Ecal(n)=0$ if $n$ is large enough.

\item If $\mathbf{H}(n)$ is bounded and non-zero, then 
$$
\mathbf{H}_\Ecal(n)=\dim \left[\Sym(\mathbb{W}^*)\otimes \Ecal\vert_{m_o}\otimes\Lcal^{\otimes n}\vert_{m_o}\right]^H
$$
for $n$ large enough. In particular the sequence $\mathbf{H}_\Ecal(n)$ is bounded.

\item If $\mathbf{H}(n)$ is bounded and $\mathbf{H}(1)\neq 0$, we have  $\mathbf{H}(n)=1,\, \forall n\geq 1$. Moreover 
 $\mathbf{H}_\Ecal(n)$ is increasing and equal to $\dim \, [\Sym(\mathbb{W}^*)\otimes \Ecal\vert_{m_o}]^H$ 
 for $n$ large enough.
\end{itemize}
}

\bigskip

In the next section we will give a consequence of Theorem A  to the branching laws between compact Lie groups.

\subsection{Stability of branching law coefficients}\label{sec:stability-branching}

Let $\rho: G\to\tilde{G}$ be a morphism between two connected compact Lie groups. We denote $d\rho: \ggot \to \tilde{\ggot}$ the induced 
Lie algebras morphism, and $\pi : \tilde{\ggot}^* \to \ggot^*$ the dual map.

Select maximal tori $T$ in $G$ and $\tilde{T}$ in $\tilde{G}$, and
Weyl chambers $\tilde{\tgot}^*_{\geq 0}$ in $\tilde{\tgot}^*$ and
$\tgot^*_{\geq 0}$ in  $\tgot^*$, where $\tgot$ and $\tilde{\tgot}$ denote the Lie algebras of $T$, resp. $\tilde{T}$.

Let $\tilde{\Lambda}_{\geq 0}\subset\tilde{\tgot}^*_{\geq 0}$, $\Lambda_{\geq 0}\subset\tgot^*_{\geq 0}$ be the
set of dominant weights. For any $(\mu,\tilde{\mu})\in \Lambda_{\geq 0}\times\tilde{\Lambda}_{\geq 0}$,
we denote $V^G_\mu$, $V^{\tilde{G}}_{\tilde{\mu}}$ the corresponding irreducible representations of $G$ and $\tilde{G}$, 
and we define
\begin{equation}\label{eq:multiplicity-K-tilde}
\mm(\mu,\tilde{\mu}) \in\N
\end{equation}
as the multiplicity of $V_{\mu}^G$ in $V^{\tilde{G}}_{\tilde{\mu}}\vert_{G}$.

For any weights $(\mu,\tilde{\mu})$, we denote $(\tilde{G}\tilde{\mu})_\mu$ the reduction of the $G$-Hamiltonian 
manifold $\tilde{G}\tilde{\mu}$ at $\mu$ : in other words $(\tilde{G}\tilde{\mu})_\mu:=
\tilde{G}\tilde{\mu}\cap\pi^{-1}(G\mu)/G$. We start with the following particular case of 
Proposition \ref{prop:QR-sjamaar} 

\begin{prop}\label{prop:theorem-B}
We have the following equivalences
\begin{itemize}
\item  $\mm(n\mu,n\tilde{\mu})=0,\ \forall n\geq 1\ \Longleftrightarrow\ (\tilde{G}\tilde{\mu})_\mu=\emptyset$

\item $\mm(n\mu,n\tilde{\mu})$ is bounded and non-zero \ $\Longleftrightarrow\ (\tilde{G}\tilde{\mu})_\mu=\{pt\}$.

\end{itemize}
\end{prop}

\medskip

When $(\tilde{G}\tilde{\mu})_\mu=\emptyset$, Theorem A tell us that for any dominant weight $(\lambda,\tilde{\lambda})$,   
$\mm(\lambda+n\mu,\tilde{\lambda} + n\tilde{\mu})=0$ when $n$ is large enough.

Let us concentrate to the case where $(\tilde{G}\tilde{\mu})_\mu\neq\emptyset$. We work here with the 
complex $G$-manifold $P=\tilde{G}\tilde{\mu}\times \overline{G\mu}$ where we take the opposite k\"{a}hler 
structure on $\overline{G\mu}$. 
Let $\xi_o\in \tilde{G}\tilde{\mu}$ 
such that $\pi(\xi_o)=\mu$: the stabilizer subgroup $H\subset G$ of $\xi_o$ is contained 
in $G_\mu$. On the coadjoint orbit $\tilde{G}\tilde{\mu}$ we work with the line bundle
$[\C_{\tilde{\mu}}]\simeq\tilde{G}\times_{\tilde{G}_{\tilde{\mu}}}\C_{\tilde{\mu}}$, and the vector bundle 
$\Ecal_{\tilde{\lambda}}:=\tilde{G}\times_{\tilde{G}_{\tilde{\mu}}}V^{\tilde{G}_{\tilde{\mu}}}_{\tilde{\lambda}}$ 
where $V^{\tilde{G}_{\tilde{\mu}}}_{\tilde{\lambda}}$ is the irreducible representation of $\tilde{G}_{\tilde{\mu}}$ 
with highest weight $\tilde{\lambda}$.

We consider the following $H$-modules associated to $p=(\xi_o,\mu)\in P$:
\begin{enumerate}
\item $\mathbb{D}:=[\C_{\tilde{\mu}}]\vert_{\xi_o}\otimes (\C_{\mu})^*\vert_H$,
\item $\mathbb{E}_{\lambda,\tilde{\lambda}}:=
\Ecal_{\tilde{\lambda}}\vert_{\xi_o} \otimes (V^{G_\mu}_{\lambda})^*\vert_H$, 
\item $\mathbb{W}:=\T_p P/\ggot_\C\cdot p$ that is isomorphic to 
$\T_{\xi_o}\tilde{G}\tilde{\mu}/\rho(\pgot_{\mu})\cdot \xi_o$. Here 
$\pgot_{\mu}$ is the parabolic sub-algebra of $\ggot_\C$ defined by $\pgot_{\mu}=\sum_{(\alpha,\mu)>0}
(\ggot_\C)_{\alpha}$.
\end{enumerate}

Note that $H^o$ acts trivially on the $H$-module $\mathbb{D}$ (it is a consequence of the Kostant relations). Hence the sequence $(\mathbb{D}^{\otimes n})_{n\geq 1}$ of $H$-modules is periodic.

\medskip

In this setting Proposition \ref{prop:equivalences} says that $(\tilde{G}\tilde{\mu})_\mu=\{pt\}$ if and only if 
the $H$-module $\Sym(\mathbb{W}^*)$ has finite $H$-multiplicities. Theorem A becomes

\medskip

\noindent{\bf Theorem B}\ {\em 
Let $(\mu,\tilde{\mu})$ be a dominant weight such that $(\tilde{G}\tilde{\mu})_\mu=\{pt\}$.

\begin{itemize}
\item We have $\mm(n\mu,n\tilde{\mu})=\dim [\mathbb{D}^{\otimes n}]^H, n\geq 1$,
and for any dominant weight $(\lambda,\tilde{\lambda})$ the equality 
$$
\mm(\lambda+ n\mu,\tilde{\lambda}+ n\tilde{\mu})= \dim [\Sym(\mathbb{W}^*)\otimes \mathbb{E}_{\lambda,\tilde{\lambda}}\otimes\mathbb{D}^{\otimes n}]^H
$$
holds for $n$ large enough. In particular the sequence $\mm(\lambda+ n\mu,\tilde{\lambda}+ n\tilde{\mu})$ 
is bounded. 

\item If $\mm(\mu,\tilde{\mu})\neq 0$, we have $\mm(n\mu,n\tilde{\mu})=1,\forall n\geq 1$. Moreover 
the sequence $\mm(\lambda+ n\mu,\tilde{\lambda}+ n\tilde{\mu})$ is increasing and constant for large 
enough $n$,  equal to $\dim [\Sym(\mathbb{W}^*)\otimes \mathbb{E}_{\lambda,\tilde{\lambda}}]^H$. 
\end{itemize}
}

\medskip

In section \ref{section:example} we give some examples where Theorem B applies.

\subsection{Stability in a non-compact case}\label{sec:stability-non-compact}

We consider here a closed subgroup $K$ of $G$ and a Hermitian $K$-module $V$. We denote $\Phi_V:V\to \kgot^*$ the (moment) map defined by $\langle\Phi_V(v),X\rangle=\frac{1}{i}(v,Xv)$. In this section we assume that the algebra $\Sym(V^*)$ of polynomial functions on $V$ has finite $K$-multiplicities.

Let $E$ be the representation of $G$ which is induced by the $K$-module $\Sym(V^*)$. We write
$E=\sum_{\mu}\mm(\mu) V^G_\mu$ where $V^G_\mu$ is the irreducible representation of $G$ parametrized by $\mu$, and  
$$
\mm(\mu)=\dim \left[ \Sym(V^*)\otimes (V^G_\mu)^*\vert_K\right]^K.
$$

The study of the asymptotic behaviour of the multiplicity function $\mu\mapsto\mm(\mu)$ uses that the representation space $E$ can be constructed as the ``geometric quantization'' of the Hamiltonian $G$-manifold 
\begin{equation}\label{eq:manifold}
M:=G\times_K(\kgot^\perp\oplus V)
\end{equation}
with moment map $\Phi$ defined by the relation
$$
\Phi([g;\xi\oplus v]):= g\left(\xi+\Phi_V(v)\right).
$$
Recall that the complex structure on $M$ comes from the natural isomorphism $M\simeq G_\C\times_{K_\C}V$.

We denote $M_\mu:=\Phi^{-1}(G\mu)/G$ the symplectic reduction of $M$ at $\mu$. Here the $[Q,R]=0$ Theorem gives the following 

\begin{prop}\label{prop:QR-sjamaar-nc}
We have the following equivalences:

$\bullet$ $\mm(n\mu)=0,\ \forall n\geq 1$ \ $\Longleftrightarrow$\  $M_\mu=\emptyset$,
 
$\bullet$ $\mm(n\mu)$ is non-zero and bounded \ $\Longleftrightarrow$ \ $M_\mu=\{pt\}$.

\end{prop}

We fix a dominant weight $\mu$. Let $m_o\in M$ such that $\Phi(m_o)=\mu$. 
Its stabilizer subgroup $H\subset G$ is contained in $G_\mu$. Hence 
the $1$-dimensional representation $\C_\mu$ of the group $G_\mu$ can be restricted to $H$. It is not difficult to see that the 
connected component $H^o$ acts trivially on $\C_\mu$. Hence the sequence $\C_{n\mu}\vert_H$ of $H$-modules is periodic. 

The complex $H$-module $\mathbb{W}:=\T_{m_o}M/\ggot_\C\cdot m_o$ associated to the point 
$m_o\simeq [g_o,v_o]\in G_\C\times_{K_\C}V$  is naturally equal to $V/\kgot_\C\cdot v_o$. 
Recall that the $H$-multiplicities in $\Sym(\mathbb{W}^*)$ are finite if and only if $\Phi^{-1}(G\mu)=Gm_o$.

\medskip

In this non-compact setting, we obtain the following stability result 

\medskip

\noindent {\bf Theorem C}{\em 
\begin{itemize}
\item If  $\mm(n\mu)=0,\ \forall n\geq 1$, then for any dominant weight $\lambda$  we have $\mm(\lambda+n\mu)=0$ if $n$ is large enough.

\item  If $\mm(n\mu)$ is bounded and non-zero, then 
$\mm(n\mu)=\dim [\C_{n\mu}\vert_H]^H, n\geq 0$,
and for any dominant weight $\lambda$
$$
\mm(\lambda+ n\mu)= \dim \left[\Sym(\mathbb{W}^*)\otimes (V_\lambda^{G_\mu})^*\vert_H\otimes\C_{- n\mu}\vert_H\right]^H
$$
for $n$ large enough. In particular the sequence $\mm(\lambda+ n\mu)$ is bounded. 
\item If $\mm(n\mu)$ is bounded and $\mm(\mu)=1$,  the sequence 
$\mm(\lambda+n\mu)$ is increasing and constant for large enough $n$,  equal to 
$\dim [S(\mathbb{W}^*)\otimes (V_\lambda^{G_\mu})^*\vert_H]^H$. 
\end{itemize}
}

\medskip

The following example recovers the situation studied in Section \ref{sec:stability-branching}

\begin{exam}
Consider the case of a morphism $\rho: K\to\tilde{K}$ between two connected compact Lie groups. 
If we work with the groups $G:=\tilde{K}\times K$, $K\croc G$ embedded diagonally, and the trivial 
module $V=0$, the $G$-manifold (\ref{eq:manifold}) corresponds to 
the cotangent bundle $\T^*\tilde{K}$ with the action of $\tilde{K}\times K$ induced by the following action of 
$\tilde{K}\times K$ on  $\tilde{K}$:  $(\tilde{k},k)\cdot g=\tilde{k}gi(k)^{-1}$. In this setting the multiplicity function 
is defined by the relation 
$$
\mm(\tilde{\lambda},\lambda) =\dim \left[V^{\tilde{K}}_{\tilde{\lambda}}\vert_{K}\otimes V^K_{\lambda}\right]^K, 
$$
for $(\tilde{\lambda},\lambda)\in \hat{\tilde{K}}\times\hat{K}$. 
\end{exam}

\section{Reduction of K\"ahler manifolds}\label{sec:Kahler}

We consider a complex manifold $M$, not necessarily compact, and  a holomorphic Hermitian line bundle $(\Lcal,\mathrm{h})$ on it. We assume that the curvature $\Omega=i(\nabla^h)^2$ of its Chern connexion $\nabla^h$ is a K\"{a}hler class (we says that the line bundle 
$\Lcal$ prequantizes the symplectic form $\Omega$).

We suppose furthermore that  a compact connected Lie group $G$ acts on $\Lcal\to M$ leaving the metric and connection invariant. 
Hence we have a moment map $\Phi:M\to\ggot^*$ defined by Kostant's relations \ref{eq:kostant-rel}. Let us assume that the $G$-action on 
$M$ extends to a $G_\C$-action and that the momentum map $\Phi$ is {\bf proper}. Then the $G$-actions on $\Lcal$ and on 
its smooth sections can both be uniquely extended to actions of $G_\C$, and the projection $\Lcal\to M$ is equivariant. 

An important object in this context is the reduced space 
$$
M_0:=\Phi^{-1}(0)/G
$$
which is compact. When $0$ is a regular value of $\Phi$, the set $M_0$ is an orbifold equipped with an 
induced K\"ahler structure  form $(\Omega_0,J_0)$, and the line orbibundle 
$\Lcal_0:=\Lcal\vert_{\Phi^{-1}(0)}/G$ prequantizes $(M_0,\Omega_0)$.

In general the set $M_0$ has a natural structure of a singular K\"{a}hler manifold that is defined as follows. 
A point $m\in M$ is (analytically) semi-stable if the closure of the $G_\C$-orbit through $m$ intersects the 
zero level set $\Phi^{-1}(0)$, and we denote the set of semi-stable points by $M^{\semi}$.

On $M^\semi$, we have a natural equivalence relation : $x\sim y\ \Longleftrightarrow \overline{G_\C x}\cap
\overline{G_\C y}\neq \emptyset$. The Mumford GIT quotient $M/\!\!/ G_\C$ is the quotient of $M^\semi$ 
by this equivalence relation  \cite{MFK,Kirwan84,Sjamaar95}.

We have the following crucial fact
\begin{theo} The set $M/\!\!/ G_\C$ has a canonical structure of a complex analytic space, and the inclusion 
$\Phi^{-1}(0)\croc M^\semi$ induces an homeomorphism $M_0\simeq M/\!\!/ G_\C$.
\end{theo}

To get a genuine line bundle on $M_0$, we have to replace $\Lcal$ by a suitable power $\Lbb:=\Lcal^{\otimes q}$ 
such that for any $m\in \Phi^{-1}(0)$ the stabilizer subgroup $G_m$ acts trivially on $\Lbb\vert_m$. Then $\Lbb_0:=
\Lcal^{\otimes q}\vert_{\Phi^{-1}(0)}/G$ is an holomorphic line bundle on $M_0$.

We need the following result (see Theorem 2.14 in \cite{Sjamaar95}).

\begin{theo}\label{theo:L_0-positive}
The line bundle $\Lbb_0$ is positive in the sense of Grauert. The reduced space $M_0$ is a complex projective variety, a projective embedding being given by the Kodaira map $M_0 \to \mathbb{P}(\Gamma(M_0, \Lbb_0^{\otimes k}))$ for all sufficiently large 
$k$.
\end{theo}

The following theorem is the first instance of the $[Q,R]=0$ phenomenon. It was proved by Guillemin-Sternberg 
\cite{Guillemin-Sternberg82} in the case where $0$ is a regular value of $\Phi$ and $M$ is compact.  In 
\cite{Sjamaar95} Sjamaar extends  their result  by dealing the non-smoothness of $M_0$ and the non-compactness of $M$. 

\begin{theo}\label{QR=0:GSS}
The quotient map $M^\semi\to M_0$ and the inclusion $M^\semi \subset M$ induce the isomorphisms 
$\Gamma(M,\Lcal)^G\simeq\Gamma(M^\semi,\Lcal)^G\simeq\Gamma(M_0,q^G_*\Lcal)$, where $q^G_*\Lcal$ is the sheaf of invariant section induces by the line bundle $\Lcal$.
\end{theo}

%

In this paper we will use Theorems \ref{theo:L_0-positive} and \ref{QR=0:GSS} to get basic results concerning the sequence $\mathbf{H}(n):=\dim \Gamma(M,\Lcal^{\otimes n})^G, \ n\geq 1$.

\begin{prop}\label{prop:HS}
For $n$ large enough, the sequence $\mathbf{H}(nq)$ is polynomial with a dominant term of the form $c n^{\alpha}$ where $\alpha$ is the complex dimension of the (smooth part of the) projective variety $M_0$.
\end{prop}

\begin{proof}
It is direct consequence of two facts: $\mathbf{H}(nq):=\dim \Gamma(M_0,\Lbb_0^{\otimes n})$ thanks 
to Theorem \ref{QR=0:GSS} and the Kodaira map $M_0 \to \mathbb{P}(\Gamma(M_0, \Lbb_0^{\otimes n}))$ is  a 
projective embedding for $n$ large enough.
\end{proof}

We get then the following useful result.

\begin{lem}\label{lem-de-base-0}
$\bullet$  $\mathbf{H}(n)=0,\ n\geq 1 \ \Longleftrightarrow \ M_0=\emptyset$.

$\bullet$ $\mathbf{H}(n)$ is non-zero and bounded $\Longleftrightarrow \ M_0=\{pt\}$. 

$\bullet$ If $\mathbf{H}(n)$ is bounded and $\mathbf{H}(1)\neq 0$, then $\mathbf{H}(n)=1$ for all $n\geq 1$.
\end{lem}

\begin{proof} The implications $\Longrightarrow$ are a consequence of Proposition \ref{prop:HS}, and the implications 
$\Longleftarrow$ are a consequence of Theorem \ref{QR=0:GSS}. For the last point we use first the $[Q,R]=0$ theorem when $M_0=\{pt\}$ : we have 
$$
\mathbf{H}(n):=\dim \left[\Lcal^{\otimes n}\vert_{m_o}\right]^H
$$
where $m\in \Phi^{-1}(0)$ and $H$ is the stabilizer subgroup of $m_o$. The $H$-module $\Lcal\vert_{m_o}$ is trivial if and only if $\mathbf{H}(1)=1$. The third point follows then.
\end{proof}

\medskip

We can now state the corresponding result that relates the multiplicities 
$$
\mm^{\Lcal}(\mu,n):= \dim \left[\Gamma(M,\Lcal^{\otimes n})\otimes (V^G_{\mu})^*\right]^G.
$$
with the reduced spaces $M_\mu:=\Phi^{-1}(G\mu)/G$.

\begin{lem}\label{lem-de-base}
$\bullet$  $\mm^{\Lcal}(n\mu,n)=0,\ n\geq 1 \ \Longleftrightarrow \ M_\mu=\emptyset$.

$\bullet$ $\mm^{\Lcal}(n\mu,n)$ is non-zero and bounded $\Longleftrightarrow \ M_\mu=\{pt\}$. 
\end{lem}

\begin{proof}
It is a direct consequence of the shifting trick by applying Lemma \ref{lem-de-base-0} to the K\"ahler manifold 
$M\times (K\mu)^-$ prequantized by the holomorphic line bundle $\Lcal\otimes [\C_{-\mu}]$.
\end{proof}

\medskip

We finish this section by recalling the following fact.

\begin{lem}\label{lem:M-0-point}
\begin{itemize}
\item Suppose that $\mathbf{H}(1)\neq 0$. Then for any holomorphic vector bundle $\Ecal\to M$,  the sequence $\mathbf{H}_\Ecal(n)=\dim \Gamma(M,\Ecal\otimes \Lcal^{\otimes n})^G$ is increasing.

\item Let $m_o\in \Phi^{-1}(0)$ with stabilizer subgroup $H$. We consider the $H$-module $\mathbb{W}:=\T_{m_o}M/\ggot_\C\cdot m_o$. Then $\Phi^{-1}(0)=Gm_o$ if and only if the algebra 
$\Sym(\mathbb{W}^*)$ has finite $H$-multiplicities.
\end{itemize}
\end{lem}

\begin{proof}
The first point follows from the fact that for any non-zero section 
$s\in \Gamma(M, \Lcal)^G$, the linear map $w\mapsto w\otimes s$ defines 
a one to one map from $\Gamma(M,\Ecal\otimes \Lcal^{\otimes n})^G$  into
$\Gamma(M,\Ecal\otimes \Lcal^{\otimes n+1})^G$.

Let us check the second point. The vector space $\ggot\cdot m_o\subset \T_{m_o} M$ is totally isotropic, 
since $ \Omega_{m_o}(X\cdot m_o,Y\cdot m_o)=\langle \Phi(m_o),[X,Y]\rangle=0$. Hence we can consider 
the vector space $E_{m_o}:=(\ggot\cdot m_o)^\perp/\ggot\cdot m_o$ that is equipped  with a 
$H$-equivariant symplectic structure $\Omega_{E_{m_o}}$ : we denote by
$\Phi_{E_{m_o}}:E_{m_o}\to \hgot^*$ the corresponding moment map.  A local model
for a symplectic neighborhood of $Gm_o$ is $G\times_{H}( \hgot^\perp\times E_{m_o})$ where the moment map is
$\Phi_{m_o}[g;\xi,v]=g(\xi+ \Phi_{E_{m_o}}(v))$. We see then that $\Phi^{-1}(0)=Gm_{o}$ if and only if the set 
$\Phi^{-1}_{E_{m_o}}(0)$ is reduced to $\{0\}$, and it is a standard fact that $\Phi^{-1}_{E_{m_o}}(0)=\{0\}$ if and only if 
the algebra $\Sym(E_{m_o}^*)$ has finite $H$-multiplicities.

We are left to prove that $E_{m_o}\simeq \mathbb{W}$. Let $J$ be a complex structure on $\T_{m_o} M$ compatible with the symplectic form $\Omega_{m_o}$. Since the vector space $\ggot_\C\cdot m_o$ is equal to the symplectic subspace
$\ggot\cdot m_o\oplus J(\ggot\cdot m_o)$, the $H$-module $\mathbb{W}$ 
has a canonical identification with its (symplectic) orthogonal 
$(\ggot\cdot m_o\oplus J(\ggot\cdot m_o))^\perp$. Finally the orthogonal decomposition 
$$
(\ggot\cdot m_o\oplus J(\ggot\cdot m_o))^\perp\oplus 
\ggot\cdot m_o=(\ggot\cdot m_o)^\perp
$$
shows that the complex $H$-modules $\mathbb{W}$ and $E_{m_o}$ are equal.

\end{proof}

\section{Witten deformation}\label{sec:witten}

\subsection{Elliptic and transversally elliptic symbols}\label{subsec:trans-elliptic}

Let us recall the basic definitions from the theory of transversally
elliptic symbols (or operators) defined by Atiyah-Singer in
\cite{Atiyah74}.

Let $M$ be a compact $G$-manifold with cotangent bundle $\T^*M$. Let $p:\T^*
M\to M$ be the projection.
If $\Ecal$ is a vector bundle on $M$, we may denote still by $\Ecal$ the vector bundle $p^*\Ecal$ on the cotangent bundle $\T^*M$.
If $\Ecal^{+},\Ecal^{-}$ are
$G$-equivariant  complex vector bundles over $M$, a
$G$-equivariant morphism $\sigma \in \Ccal^{\infty}(\T^*
M,\hatom(\Ecal^{+},\Ecal^{-}))$ is called a {\em symbol} on $M$.
 For $x\in  M$, and $\nu\in T_x^*M$,  thus $\sigma(x,\nu):\Ecal\vert_x^{+}\to
\Ecal\vert_x^{-}$
is a linear map.
The
subset of all $(x,\nu)\in \T^* M$ where the map $\sigma(x,\nu)$  is not invertible is called the {\em characteristic set}
of $\sigma$, and is denoted by $\Char(\sigma)$.
A symbol  is elliptic if its characteristic set is compact.

%
%
%
%
%
%
%
%
%

The product of a symbol  $\sigma$
by a  $G$-equivariant complex vector bundle $\Fcal\to M$ is the symbol
$\sigma\otimes \Fcal$
defined by
$$(\sigma\otimes \Fcal)(x,\nu)=\sigma(x,\nu)\otimes 1_{\Fcal_x}.$$
%
%
%

 An elliptic symbol $\sigma$  on $M$ defines an
element $[\sigma]$ in the equivariant $\K$-theory of $\T^*M$ with compact
support, which is denoted by $\Ko_{G}(\T^* M)$. 
The
index of $\sigma$ is a virtual finite dimensional representation of
$G$, that we denote by $\indice_{G}^M(\sigma)\in R(G)$
\cite{Atiyah-Segal68,Atiyah-Singer-1,Atiyah-Singer-2,Atiyah-Singer-3}.

Recall the notion of {\it transversally elliptic symbol}.
Let  $\T^*_G M$ be the following $G$-invariant closed subset of $\T^*M$
$$
   \T^*_{G}M\ = \left\{(x,\nu)\in \T^* M,\ \langle \nu,X\cdot x\rangle=0 \quad {\rm for\ all}\
   X\in\ggot \right\} .
$$
 Its fiber over a point $x\in M$ is  formed by all the cotangent vectors $v\in T^*_xM$  which vanish on the tangent space to the orbit of $x$  under $G$, in the point $x$. Thus   each fiber $(\T^*_G M)_x$ is a linear subspace  of $T_x^* M$. In general the dimension of $(\T^*_GM)_x$  is not constant and this space is not a vector bundle.
A symbol $\sigma$ is  $G$-{\em transversally elliptic} if the
restriction of $\sigma$ to $\T^*_{G} M$ is invertible outside a
compact subset of $\T^*_{G} M$ (i.e. $\Char(\sigma)\cap
\T_{G}^*M$ is compact).

A $G$-{\em transversally elliptic} symbol $\sigma$ defines an
element of $\Ko_{G}(\T^*_{G}M)$, and the index of
$\sigma$
defines an element $\indice_G^M(\sigma)$ of $\hat{R}(G)$.

The index map $\indice_{G}^M: \Ko_K(\T_G^*M)\to \hat R(G)$ is a morphism of $R(G)$ module:  for any $G$-module $V$,
\begin{equation}\label{eq:indextimesVmu}
\indice_{G}^M(\sigma\otimes V)= \indice_{G}^M(\sigma)\otimes V.
\end{equation}

Any elliptic symbol  is $G$-transversally
elliptic, hence we have a restriction map
$\K_{G}^0(\T^* M)\to \Ko_{G}(\T_{G}^*M)$, and a commutative
diagram
\begin{equation}\label{indice.generalise}
\xymatrix{ \Ko_{G}(\T^* M) \ar[r]\ar[d]_{\indice_G^M}
&
\Ko_{G}(\T_{G}^*M)\ar[d]^{\indice_G^M}\\
R(G)\ar[r] & \Rfor(G)\ .
   }
\end{equation}

\medskip

Using the {\em excision property}, one can easily show that the
index map $\indice_G^\Ucal: \K_{G}^0(\T_{G}^*\Ucal)\to
\Rfor(G)$ is still defined when $\Ucal$ is a
$G$-invariant relatively compact open subset of a
$G$-manifold (see \cite{pep-RR}[section 3.1]).

Finally the index map $\indice_G^M: \K_{G}^0(\T_{G}^*M)\to
\Rfor(G)$ can be still defined when $M$ is a {\em non-compact} manifold. Any class $\sigma\in \K_{G}^0(\T_{G}^*M)$ is represented
by a symbol on $M$ with a characteristic set $\Char(\sigma)\subset \T^* M$ intersecting $\T^*_G M$ in a compact set. Let
$\Ucal$ be a $G$-invariant relatively compact open subset of $M$ such that $\Char(\sigma)\cap \T^*_G M\subset \T^*\Ucal$. The restriction
$\sigma\vert_\Ucal$ defines a $G$-transversally elliptic symbol on $\Ucal$, and we take
$$
\indice_G^M(\sigma):=\indice_G^\Ucal(\sigma\vert_\Ucal).
$$


{\bf Remark :} In the following the manifold $M$ will carry a $G$-invariant Riemannian metric and we will denote by
$\nu\in \T^*M \mapsto \tilde{\nu}\in \T M$ the corresponding identification.

\subsection{Localization of the Riemann-Roch character}\label{sec:localisation}

Let $M$ be a $G$-manifold equipped with an invariant almost complex structure $J$. Let $p:\T M\to M$ be the projection. The complex vector bundle $(\T^* M)^{0,1}$ is $G$-equivariantly identified with the tangent bundle $\T M$ equipped with the complex structure $J$. Let $h_M$ be an  Hermitian structure on  $(\T M,J)$. The symbol 
$\Thom(M,J)\in 
\Ccal^{\infty}\left(\T^* M,\hatom(p^{*}(\wedge_{\C}^{even} \T M),\,p^{*}
(\wedge_{\C}^{odd} \T M))\right)$  
at $(m,\nu)\in \T M$ is equal to the Clifford map
\begin{equation}\label{eq.thom.complex}
 \clif_{m}(\nu)\ :\ \wedge_{\C}^{even} \T_m M
\longrightarrow \wedge_{\C}^{odd} \T_m M,
\end{equation}
where $\clif_{m}(\nu).w= \tilde{\nu}\wedge w - \iota(\tilde{\nu})w$ for $w\in 
\wedge_{\C}^{\bullet} \T_{m}M$. Here $\iota(\tilde{\nu}):\wedge_{\C}^{\bullet} 
\T_{m}M\to\wedge_{\C}^{\bullet -1} \T_{m}M$ denotes the 
contraction map relative to $h_M$. Since $\clif_{m}(\nu)^2=-\|\nu\|^2 {\rm Id}$, the map  
$\clif_{m}(\nu)$ is invertible for all $\nu\neq 0$. Hence the symbol $\Thom(M,J)$ is elliptic when the manifold $M$ is compact. 

\begin{defi}
Suppose that $M$ is compact. To any $G$-equivariant complex vector bundle $\Ecal\to M$, we associate its Riemann-Roch character 
$$
\RR^J_G(M,\Ecal):=\indice_G^M(\Thom(M,J)\otimes \Ecal)\in R(G).
$$
If the complex structure $J$ is understood we just denote $\RR_G(M,-)$ the Riemann-Roch character.
\end{defi}

\begin{rem}
The character $\RR_G(M, \Ecal)$ is equal to the equivariant index of the Dolbeault-Dirac operator $\Dcal_\Ecal:=\sqrt{2}(\overline{\partial}_{\Ecal} + \overline{\partial}^*_{\Ecal})$, since $\Thom(M,J)\otimes \Ecal$  corresponds to the 
principal symbol  of $\Dcal_\Ecal$ (see \cite{B-G-V}[Proposition 3.67]).  
\end{rem}

\medskip

Let us briefly explain how we perform the ``Witten''  deformation the symbol $\Thom(M,J)$ with the help of 
an equivariant map $\phi:M\to \ggot^*$ \cite{pep-RR,Ma-Zhang14,pep-vergne:witten}. 
Consider the identification $\xi\mapsto\wtde{\xi},\ggot^*\to\ggot$ defined by a $G$-invariant scalar product
on  $\ggot^*$. We define the {\em Kirwan vector field}: 
\begin{equation}\label{eq-kappa}
    \kappa_\phi(m)= \left(\wtde{\phi(m)}\right)_M(m), \quad m\in M.
\end{equation}

We denote $Z_\phi\subset M$ the subset where $\kappa_\phi$ vanishes.

\begin{defi}\label{def:pushed-sigma}
The symbol  $\Thom(M,J)$ pushed by the vector field $\kappa_\phi$ is the symbol $\clif_\phi$ 
defined by the relation
$$
\clif_\phi\vert_m(\nu)=\Thom(M,J)\vert_m(\tilde{\nu}-\kappa_\phi(m))
$$
for any $(m,v)\in\T^* M$. 
\end{defi}

Note that $\clif_\phi\vert_m(\nu)$ is invertible except if
$\tilde{\nu}=\kappa_\phi(m)$. If furthermore $\nu$ belongs to the subset $\T^*_G M$
of cotangent vectors orthogonal to the $G$-orbits, then $\nu=0$ and
$m\in Z_\phi=\{\kappa_\phi=0\}$.  Indeed $\kappa_\phi(m)$ is tangent to $G\cdot m$ while
$\nu$ is orthogonal. Finally we have 
$\Char(\clif_\phi)\cap \T_G^* M \simeq
Z_\phi$.

\begin{defi}
When the critical set $Z_\phi$ is compact, we define \break 
$\RR_G(M,\Ecal,\phi)\ \in\ \hat{R}(G)$ as the equivariant index of the transversally 
elliptic symbol $\clif_\phi\otimes \Ecal\in \Ko_G(\T_G^* M)$.
\end{defi}

When $M$ is compact, it is clear that the classes of the symbols $\clif_\phi\otimes \Ecal$ and
$\Thom(M,J)\otimes \Ecal$ are equal in $\K_{G}^0(\T_{G}^*M)$, hence the equivariant indices 
$\RR_G(M,\Ecal)$ and $\RR_G(M,\Ecal,\phi)$ are equal.

For any $G$-invariant open subset $U\subset M$ such that $U\cap Z_\phi$ is compact in $M$, we see that the restriction
$\clif_\phi\vert_U$ is a transversally elliptic symbol on $U$, and so its equivariant index is a well defined element in
$\hat{R}(G)$.

\begin{defi}\label{def:indice-localise}
$\bullet$ A closed invariant subset $Z\subset Z_\phi$ is called a component 
if it is a union of connected components of $Z_\phi$.

$\bullet$ For a compact component $Z$ of $Z_\phi$, we denote by
$$
\RR_G(M,\Ecal,Z,\phi)\ \in\ \hat{R}(G)
$$
the equivariant index of $\clif_\phi\otimes \Ecal\vert_{\T^* U}$, where $U$ is any $G$-invariant open subset 
 such that $U\cap \{\kappa_{\phi}=0\}=Z$. By definition, $\RR_G(M,\Ecal,Z,\phi)=0$ when $Z=\emptyset$.
\end{defi}

In this paper we will have a particular interest to the character
$$
\RR_G(M,\Ecal,\phi^{-1}(0),\phi)\in \hat{R}(G).
$$
which is defined when $\phi^{-1}(0)$ is a compact component of $Z_\phi$.

\subsection{$[Q,R]=0$}\label{sec:QR=0}

When $(M,\Omega,\Phi)$ is a compact Hamiltonian $G$-manifold, the Riemann-Roch character $\RR_G(M,-)$ is computed 
with an invariant almost complex structure $J$ that is compatible with $\Omega$. Here the Kirwan vector field 
$\kappa_\Phi$ is the Hamiltonian vector field of the function $\frac{-1}{2}\|\Phi\|^2$. Hence the set $Z_\Phi$ 
of zeros of $\kappa_\Phi$ coincides with the set of critical points of $\|\Phi\|^2$. When $M$ is non compact but 
the critical set $Z_\Phi$ is compact, we can define the localized Riemann-Roch character $\RR_G(M,-,\Phi)$. If moreover the map $\Phi$ is proper, the set $\Phi^{-1}(0)$ will be a compact component of $Z_\Phi$, so we can consider the 
localized Riemann-Roch character $\RR_G(M,-,\Phi^{-1}(0),\Phi)$.

Let $\Lcal\to M$ be a Hermitian line bundle that prequantizes the data $(M,\Omega,\Phi)$.  In this setting 
we are interested by the  dimension of the trivial $G$-representation in $\RR_G(M,\Lcal^{\otimes n})$ 
that we simply denote $[\RR_G(M,\Lcal^{\otimes n})]^G\in\Z$.

One of the main fact of this localization procedure is the following

\begin{theo}[\cite{pep-RR, pep-vergne:witten}]\label{th:QR=0}
Let $(M,\Omega,\Phi)$ be a Hamiltonian $G$-manifold prequantized by a line bundle $\Lcal$. Let $\Ecal$ be an equivariant vector bundle on $M$.

$\bullet$ When $M$ is compact, we have 
\begin{eqnarray*}
\left[\RR_G(M,\Lcal^{\otimes n})\right]^G&=&\left[\RR_G(M,\Lcal^{\otimes n},\Phi^{-1}(0),\Phi)\right]^G,
\ \mathrm{for}\ n\geq 1, \\
\left[\RR_G(M,\Lcal^{\otimes n}\otimes\Ecal)\right]^G&=&\left[\RR_G(M,\Lcal^{\otimes n}\otimes\Ecal,\Phi^{-1}(0),\Phi)\right]^G,\ \mathrm{for}\  n >\!> 1.
\end{eqnarray*}

$\bullet$ If $\Phi$ is proper and the critical set $Z_\Phi$ is compact,  we have 
\begin{eqnarray*}
\left[\RR_G(M,\Lcal^{\otimes n},\Phi)\right]^G&=&\left[\RR_G(M,\Lcal^{\otimes n},\Phi^{-1}(0),\Phi)\right]^G,
\ \mathrm{for}\ n\geq 1, \\
\left[\RR_G(M,\Lcal^{\otimes n}\otimes\Ecal,\Phi)\right]^G&=&\left[\RR_G(M,\Lcal^{\otimes n}\otimes\Ecal,\Phi^{-1}(0),\Phi)\right]^G,
\ \mathrm{for}\  n >\!> 1.
\end{eqnarray*}

\end{theo}

\medskip

Let us finish this section by explaining the cases where the quantity 
$\left[\RR_G(M,\Ecal,\Phi^{-1}(0),\Phi)\right]^G$ can be computed as an 
index on the reduced space $M_0$.

First suppose that $0$ is a regular value of $\Phi$. The reduced space $M_0$ is a symplectic orbifold, and we 
can define in this context a Riemann-Roch character $\RR(M_0,-)$ with the help of a compatible almost complex structure. For any equivariant vector bundle $\Fcal$ on $M$ we define the orbibundle $\Fcal_0:=\Fcal\vert_{\Phi^{-1}(0)}/G$ on $M_0$, and we have 
$$
\left[\RR_G(M,\Fcal,\Phi^{-1}(0),\Phi)\right]^G=\RR(M_0,\Fcal_0).
$$

Suppose now that $0$ is a quasi-regular value of $\Phi$. It is the case when there exists a sub-algebra 
$\hgot$  of $\ggot$  such that $Z:=\Phi^{-1}(0)$ is contained in the sub-manifold $M_{(\hgot)}=GM_{\hgot}$ 
where $M_\hgot=\{m\in M, \ggot_m=\hgot\}$. Let $N$ be the normalizer subgroup of $\hgot$ in $G$, 
and let $H^o$ be the closed connected subgroup of $G$ with Lie algebra $\hgot$.  Thus 
$M_{(\hgot)}\simeq G\times_N M_\hgot$ and $Z\simeq G\times_N Z_\hgot$ where 
$Z_\hgot:=\Phi^{-1}(0)\cap M_\hgot$ is a compact $N$-submanifold of $M$ with a locally 
free action of $N/H^o$. Then the reduced space
$$
M_{0}:=\Phi^{-1}(0)/G= Z_\hgot/ (N/H^o)
$$
is a compact symplectic orbifold.

Let $\Wcal\to Z$ be the symplectic normal bundle of the submanifold $Z$ in $M$: for $x\in Z$,
$$
\Wcal\vert_x= (\T_x Z)^{\perp}/(\T_x Z)^{\perp}\cap\T_x Z,
$$
were we have denoted by $(\T_x Z)^{\perp}$ the orthogonal with respect to the symplectic form. We can 
equip $\Wcal$ with an $H$-invariant Hermitian structure $\mathrm{h}$ such that the symplectic structure on the fibres
of $\Wcal\to Z$ is equal to $-\mathrm{Im}(\mathrm{h})$.

  The sub-algebra $\hgot$ acts fiberwise on the complex vector bundle  $\Wcal\vert_{Z_\hgot}$. We consider 
  the action of $\hgot$ on the fibres of the complex bundle $\Sym(\Wcal^*\vert_{Z_\hgot})$. We will use the following result
  (\cite{pep-vergne:witten}[Section 12.2]).

\begin{lem}\label{lem:sym}
The sub-bundle $[\Sym(\Wcal^*\vert_{Z_\hgot})]^\hgot$ is reduced to the trivial bundle $[\C]\to Z_\hgot$.
\end{lem}

Thanks to Lemma \ref{lem:sym}, we can introduce the following notion of reduction in the quasi-regular case.

\begin{defi}
If $\Fcal\to M$ is a $K$- equivariant complex vector bundle, we define on $M_{0}$ the (finite dimensional)  orbi-bundle
$$
\Fcal_{0}:=\left[\Fcal\vert_{Z_\hgot}\otimes \Sym(\Wcal^*\vert_{Z_\hgot})\right]^\hgot/(N/H^o).
$$
If $\hgot$ acts trivially on the fibres of  $\Fcal\vert_{Z_\hgot}$, the bundle $\Fcal_{0}$ is equal to
$\Fcal\vert_{Z_\hgot}/(N/H^o)$.
\end{defi}

The following result is proved in \cite{pep-vergne:witten}[Section 12.2].
\begin{theo}\label{theo: Qsemiregular}
Assume that $\Phi^{-1}(0)\subset M_{(\hgot)}$. For any $G$-equivariant complex vector bundle $\Fcal\to M$, we have
 $$
\left[\RR_G(M,\Fcal, \Phi^{-1}(0),\Phi)\right]^G=\RR(M_{0},\Fcal_{0}).
 $$
\end{theo}

A case of particular interest for us is when the reduced space $M_{0}:=\Phi^{-1}(0)/G$ is reduced to a point : we are in the quasi-regular case. Let $H$ be the stabilizer subgroup of $m_o\in Z:=\Phi^{-1}(0)$ (which is not necessarily connected). Then 
$Z=G\cdot m_o\simeq G/H$ is contained in $GM_\hgot$ where $\hgot$ is the Lie algebra of $H$. 

By definition the fiber of the complex vector bundle $\Wcal\to Z$ at $m_o$ is 
$\Wcal\vert_{m_o}=(\ggot\cdot m_o)^\perp/ \ggot\cdot m_o$. We have checked in the proof of 
Lemma \ref{lem:M-0-point} that the $H$-modules $\Wcal\vert_{m_o}$ coincides with 
$\mathbb{W}:=\T_{m_o}M/\ggot_\C\cdot m_o$. Recall that the hypothesis $Z=G\cdot m_o$ 
is equivalent to the fact that the complex $H$-module $\Sym(\mathbb{W}^*)$ has finite multiplicities.

In this case Theorem \ref{theo: Qsemiregular} gives 

\begin{coro}\label{coro:point}
Suppose that $\Phi^{-1}(0)=G\cdot m_o$ with $G_{m_o}=H$. For any $G$-equivariant complex vector bundle $\Fcal\to M$, we have
 $$
\left[\RR_G(M,\Fcal, \Phi^{-1}(0),\Phi)\right]^G=\left[\Sym(\mathbb{W}^*)\otimes \Fcal\vert_{m_o}\right]^H,
 $$
 where $\mathbb{W}:=\T_{m_o} M/\ggot_\C\cdot m_o$.
\end{coro}

\subsection{Main proofs}\label{sec:proofs}

\subsubsection{Proof of Theorem {\bf A}}\label{sec:proof-theorem-A}

Consider a $G$-compact complex manifold $M$ endowed with an ample holomorphic $G$-line 
$\Lcal\to M$ with curvature the symplectic two-form $\Omega$. Let $\Phi:M\to\ggot^*$ be the moment map 
associated to the $G$-action on $\Lcal$ (see (\ref{eq:kostant-rel})). 

Let $\Ecal\to M$ be an holomorphic $G$-vector bundle. In this context, we are interested in the family of 
$G$-modules $\Gamma(M,\Ecal\otimes\Lcal^{\otimes n})$ formed by the holomorphic sections. We 
denote $\mathbf{H}_\Ecal(n)$ the dimension of $\Gamma(M,\Ecal\otimes\Lcal^{\otimes n})^G$. When 
we take $\Ecal=\C$, we denote $\mathbf{H}(n)=\dim \Gamma(M,\Lcal^{\otimes n})^G$.

By Kodaira vanishing theorem, we know that 
$$
\mathbf{H}_\Ecal(n)=[\RR_G(M,\Ecal\otimes\Lcal^{\otimes n})]^G
$$
when $n$ is sufficiently large. On the other hand Theorem \ref{th:QR=0} tell us that 
$[\RR_G(M,\Ecal\otimes\Lcal^{\otimes n})]^G$ is equal to 
$\left[\RR_{G}(M,\Ecal\otimes\Lcal^{\otimes n},\Phi^{-1}(0),\Phi)\right]^G$
for $n$ large enough. We know then that 
\begin{equation}\label{eq:H-E}
\mathbf{H}_\Ecal(n)=\left[\RR_{G}(M,\Ecal\otimes\Lcal^{\otimes n},\Phi^{-1}(0),\Phi)\right]^G
\end{equation}
when $n$ is sufficiently large. Two cases are considered in Theorem {\bf A}.

\medskip

$\bullet$ Suppose that $\mathbf{H}(n)=0$ for all $n\geq 1$. We have seen in Lemma \ref{lem-de-base-0} that it means that $\Phi^{-1}(0)=\emptyset$. In this case relation (\ref{eq:H-E}) shows that $\mathbf{H}_\Ecal(n)=0$ if $n$ is large enough.

\medskip

$\bullet$ Suppose that the sequence $\mathbf{H}(n)$ is non-zero and bounded: here we have that $\Phi^{-1}(0)=G\cdot m_o$ for some $m_o\in M$. Corollary \ref{coro:point} tell us that 
$$
\left[\RR_{G}(M,\Ecal\otimes\Lcal^{\otimes n},\Phi^{-1}(0),\Phi)\right]^G=
\left[\Sym(\mathbb{W}^*)\otimes \Ecal\vert_{m_o}\otimes\Lcal^{\otimes n}\vert_{m_o}\right]^H,
$$
where $H$ is the stabilizer subgroup of $m_o$, and $\mathbb{W}:=\T_{m_o} M/\ggot_\C\cdot m_o$.

The proof of Theorem  {\bf A} is then completed.

\subsubsection{Proof of Theorem {\bf B}}\label{sec:proof-theorem-B}

Here we use the notations of Section \ref{sec:stability-branching}. We fix a dominant weight 
$(\mu,\tilde{\mu})$ for the group $G\times \tilde{G}$, and we work with the $G$-manifold
$P=\tilde{G}\tilde{\mu}\times (G\mu)^-$, where $(G\mu)^-$ is the coadjoint orbit with the opposite symplectic  
and complex structure. The line bundle $\Lcal_P:=[\C_{\tilde{\mu}}]\boxtimes [\C_{-\mu}]$ prequantizes the symplectic form $\Omega_P:=\Omega_{\tilde{G}\tilde{\mu}}\times -\Omega_{G\mu}$. The moment map 
$\Phi_P: M\to \ggot^*$ is  defined by the relation $\Phi_P(\tilde{\xi},\xi)=\pi(\tilde{\xi})-\xi$.

The Borel-Weil Theorem says that the $G$-module $\Gamma(P,\Lcal_P^{\otimes n})$ corresponds to the tensor product 
$ (V^G_{n\mu})^*\otimes V^{\tilde{G}}_{n\tilde{\mu}}\vert_G$, hence  
$\mathbf{H}(n):=\dim [\Gamma(M,\Lcal^{\otimes n})]^G$ is equal to the multiplicity $\mm(n\mu,n\tilde{\mu})$. Here Lemma \ref{lem-de-base-0} tell us that the sequence $\mm(n\mu,n\tilde{\mu})$ is bounded if and only if the reduced space 
$(\tilde{G}\tilde{\mu})_\mu$ is empty or reduced to a singleton.

Now we want to investigate the behaviour of the sequence $\mm(\lambda+ n\mu,\tilde{\lambda}+n\tilde{\mu})$. 
On the coadjoint orbit $\tilde{G}\tilde{\mu}$ (resp. $G\mu$) we consider the complex vector bundle 
$\Ecal_{\tilde{\lambda}}:= \tilde{G}\times_{\tilde{G}_{\tilde{\mu}}}V_{\tilde{\lambda}}^{\tilde{G}_{\tilde{\mu}}}$ 
(resp. $\Ecal_\lambda:= G\times_{G_\mu}V_\lambda^{G_\mu}$). A direct computation gives that 
$V_{\tilde{\lambda}+ n\tilde{\mu}}^{\tilde{G}}=\RR_{\tilde{G}}(\tilde{G}\tilde{\mu}, \Ecal_{\tilde{\lambda}}
\otimes [\C_{\tilde{\mu}}]^{\otimes n})$ and $(V_{\lambda+ n\mu}^{G})^*=\RR_{G}((G\mu)^-, \Ecal_{\lambda}^*
\otimes [\C_-\mu]^{\otimes n})$, hence
$
\mm(\lambda+n\mu,\tilde{\lambda}+n\tilde{\mu})=
\left[\RR_G(M,\Ecal_{\tilde{\lambda}}\otimes \Ecal_\lambda^*\otimes \Lcal^{\otimes n})\right]^G$.

\medskip

If we use Theorem A, we obtain the following result:

$\bullet$ Suppose that $\mm(n\mu,n\tilde{\mu})=\mathbf{H}(n)=0$ for all $n\geq 1$. Then $\mm(\lambda+n\mu,\tilde{\lambda}+n\tilde{\mu})=\mathbf{H}_\Ecal(n)=0$ if $n$ is large enough.

\medskip

$\bullet$ Suppose that the sequence $\mm(n\mu,n\tilde{\mu})=\mathbf{H}(n)$ is non-zero and bounded: here we have that $\Phi^{-1}(0)=G(\xi_o,\mu)$ for some $\xi_o\in\tilde{G}\tilde{\mu}$. In this case we have 
$$
\mm(\lambda+n\mu,\tilde{\lambda}+n\tilde{\mu})=
\left[\Sym((\mathbb{W}^P)^*)\otimes \mathbb{E}_{\lambda,\tilde{\lambda}}\otimes 
\mathbb{D}^{\otimes n}\right]^H,
$$
if $n$ is large enough. Here $H$ is the stabilizer subgroup of $p=(\xi_o,\mu)$, $\mathbb{W}^P=\T_p P/\ggot_\C\cdot p$, $\mathbb{D}:=[\C_{\tilde{\mu}}]\vert_{\xi_o}\otimes (\C_{\mu})^*\vert_H$, and $\mathbb{E}_{\lambda,\tilde{\lambda}}:=
\Ecal_{\tilde{\lambda}}\vert_{\xi_o} \otimes (V^{G_\mu}_{\lambda})^*\vert_H$.

The proof of Theorem B is completed with the following 

\begin{lem}
The $H$-modules $\mathbb{W}^P$ is isomorphic to $\T_{\xi_o}\tilde{G}\tilde{\mu}/\rho(\pgot_{\mu})\cdot \xi_o$, 
where $\pgot_{\mu}$ is the parabolic sub-algebra of $\ggot_\C$ defined by 
$\pgot_{\mu}=\sum_{(\alpha,\mu)\geq 0}(\ggot_\C)_{\alpha}$.
\end{lem}

\begin{proof} We have $\T_p P\simeq \T_{\xi_o}\tilde{G}\tilde{\mu}\times \ggot_\C/\pgot_{\mu}$. Hence the map 
$\T_{\xi_o}\tilde{G}\to \T_p P/\ggot_\C\cdot p, v\mapsto \overline{(v,0)}$ is surjective with kernel equal to 
$\rho(\pgot_{\mu})\cdot \xi_o$.
\end{proof}

\subsubsection{Proof of Theorem {\bf C}}\label{sec:proof-theorem-C}

Here $K$ is a closed subgroup of $G$, and we use a $K$-invariant decomposition~: $\ggot=\kgot\oplus\qgot$. Let 
$V$ be a $K$-Hermitian vector space such that the $K$-module $\Sym(V^*)$ has finite multiplicities. The proof of Theorem \ref{theo:principal} is an adaptation of the previous arguments to the case where we work with the non-compact manifold $M:=G\times_K(\qgot^*\oplus V)$.

The symplectic structure on $M$ is defined as follows. Let $\theta\in \Acal^1(G)\otimes \ggot$ the canonical 
connection relatively to right translation~: $\theta(\frac{d}{dt}\vert_{t=0} ge^{tX})= X$. Let $\Omega_V$ be 
the symplectic structure on $V$ which is $-1$ times the imaginary part of the hermitian structure of $V$. Let $\lambda_V$ the invariant $1$-form on 
$V$ defined by $\lambda_V(v)=\frac{1}{2}\Omega_V(v,-)$ : we have $\Omega_V=d\lambda_V$. The moment map $\Phi_V:V\to \kgot^*$  
associated to the $K$-action on $(V,\Omega_V)$ is defined by $\langle\Phi_V(v),X\rangle=\frac{1}{2}\Omega_V(Xv,v)$. Recall that our hypothesis 
``the $K$-module $\Sym(V^*)$ has finite multiplicities'' implies that the map $\Phi_V$ is proper: one has a relation of the form $\|\Phi_V(v)\|\geq c\|v\|^2$ for some $c>0$.

We consider the $1$-form $\lambda:=
\lambda_V- \langle \xi\oplus\Phi_V,\theta\rangle$ on $G\times(\qgot^*\oplus V)$, which is 
$G \times K$-equivariant and $K$-basic. It induces a $1$-form $\lambda_M$ on $M$.

We have the standard fact.
\begin{prop}

$\bullet$ The $2$-form $\Omega_M:=d\lambda_M$ defines a $G$-invariant symplectic form on $M$. 
The corresponding moment map is
$\Phi([g;\xi\oplus v])= g(\xi\oplus\Phi_V(v))$.

$\bullet$ The moment map $\Phi$ is proper and $Z_{\Phi}\simeq G/K$.

$\bullet$ The trivial line bundle $\C$ on $M$ prequantizes the $2$-form $\Omega_M$.
\end{prop}

We equip $M$ with an invariant almost complex structure compatible with $\Omega_M$. Since the critical 
set $Z_{\Phi}$ is compact, one can define the the localized Riemann-Roch character  $\RR_G(M,-,\Phi)$.  
The following result is proved in \cite{pep-fomal2}[Section 2.3].

\begin{prop}\label{prop-RR-Phi}
We gave
$$
\RR_G(M,\C,\Phi)=\mathrm{Ind}_K^G\left(\Sym(V^*)\right)=\sum_{\mu\in\hat{G}} \mm(\mu) V^G_\mu,
$$
with $\mm(\mu)=\dim \left[ \Sym(V^*)\otimes (V^G_\mu)^*\vert_K\right]^K$.
\end{prop}

In order to compute geometrically $\mm(\mu)$ we have to adapt the shifting trick to this non-compact setting. 
Let us fix two dominant weight $\mu$ and $\lambda$. Like in the previous section we work with the 
$G$-manifold $P=M\times (G\mu)^-$, that is equipped with 

$\bullet$ the symplectic form  $\Omega_P:=\Omega_M\times -\Omega_{G\mu}$,

$\bullet$ the line bundle $\Lcal_P:=\C\boxtimes [\C_{\mu}]^{-1}$ that prequantizes $\Omega_P$,

$\bullet$ the moment map $\Phi_P: P\to \ggot^*$ that is  defined by the relation $\Phi_P(m,\xi)=\Phi_M(m)-\xi$,

$\bullet$ the vector bundle $\Ecal_\lambda:= \C\boxtimes G\times_{G_\mu}V_\lambda^{G_\mu}$.

For any $R\geq 0$, we define $M_{\leq R}$ as the compact subset of points 
$[g;\xi\oplus v]$ such that $\|\xi\|\leq R$ and $\| v\|\leq R$. We start with the following basic fact whose 
proof is left to the reader.

\begin{lem}\label{lem:crit-Z-Phi}
There exists $c>0$, such that for any $\mu$ the critical set $Z_{\Phi_P}\subset P=M\times G\mu$ is 
contained in the compact set 
$M_{\leq c\|\mu\|}\times G\mu$.
\end{lem}

Since $Z_{\Phi_P}$ is compact we can consider the localized Riemann-Roch character $\RR_G(P,-,\Phi_P)$. Here the map 
$\Phi_P$ is also proper, hence we can consider the Riemann-Roch character $\RR_G(P,-,\Phi_P^{-1}(0),\Phi_P)$ localized 
on the compact component $\Phi_P^{-1}(0)$.  

\begin{lem}We have 
$$
\mm^\Lcal(\lambda+ n\mu,n)=\left[\RR_{G}(P,\Ecal_\lambda^* \otimes \Lcal_P^{\otimes n},\Phi_P^{-1}(0),\Phi_P)\right]^G
$$
for $n$ large enough.
\end{lem}

\begin{proof} We consider the family of equivariant maps $\phi^t: P\to \ggot^*, t\in [0,1]$ defined by the relation $\phi^t(m,\xi)=\Phi_M(m)-t\xi$. 
Let $\kappa^t$ be the Kirwan vector field attached to $\phi^t$, and let $Z_{\phi^t}$ be the vanishing set of $\kappa^t$: thanks to Lemma 
\ref{lem:crit-Z-Phi} we know that $Z_{\phi^t}$ is a compact subset included in $M_{c\|\mu\|}\times G\mu$ for any $t\in [0,1]$.

We know then that the family of pushed symbols $\clif_{\phi^t}$ is an homotopy of transversally elliptic symbols on $P$. We get then that 
\begin{eqnarray*}
\RR_{G}(P,\Ecal_\lambda^*\otimes\Lcal_P^{\otimes n},\Phi_P)&=&\RR_{G}(P,\Ecal_\lambda^*\otimes\Lcal_P^{\otimes n},\phi^0)\\
&=&\RR_G(M,\C,\Phi_M)\otimes \RR_G((G\mu)^-,\Ecal_\lambda^*\otimes [\C_{-n\mu}])\\
&=&\RR_G(M,\C,\Phi_M)\otimes (V_{\lambda+n\mu}^G)^*.
\end{eqnarray*}
At this stage we have proved that 
\begin{equation}\label{eq:1}
\mm(\lambda+ n\mu)=\left[\RR_{G}(P,\Ecal_\lambda^*\otimes\Lcal_P^{\otimes n},\Phi_P)\right]^G
\end{equation}
for any $n\geq 0$. Since Theorem \ref{th:QR=0} tells us that the right hand side of (\ref{eq:1}) is equal to 
$\left[\RR_{G}(P,\Ecal_\lambda^*\otimes\Lcal_P^{\otimes n},\Phi_P^{-1}(0),\Phi_P)\right]^G$ for large enough $n$, 
the proof of our Lemma is completed.
\end{proof}

\medskip

Like in the previous section, the term 
$$
Q_{\lambda,\mu}(n):=\left[\RR_{G}(P,\Ecal_\lambda^*\otimes\Lcal_P^{\otimes n},\Phi_P^{-1}(0),\Phi_P)\right]^G
$$
can be computed explicitly when the reduced space $\Phi_P^{-1}(0)/G$ is empty or a point:

$\bullet$ If $\Phi_P^{-1}(0)=\emptyset$, we have $Q_{\lambda,\mu}(n)=0$ for any $n\geq 0$.

$\bullet$ If $\Phi_P^{-1}(0)= G\cdot(m_o,\mu)$ for some $m_o\in M$, we have 
$$
Q_{\lambda,\mu}(n)=\left[\Sym(\mathbb{W}^*)\otimes \C_{-n\mu}\otimes (V_\lambda^{G_\mu})^*\vert_H\right]^H, \quad n\geq 0
$$
where $H$ is the stabilizer subgroup of $m_o$, and $\mathbb{W}=\T_{m_o}M/ \pgot_\mu\cdot m_o$.

\medskip

We have proved that :

$\bullet$ if $M_\mu=\emptyset$, we have $\mm(\lambda+ n\mu)=0$ if $n$ is large enough, for any dominant 
weight $\lambda$,

$\bullet$ if $M_\mu=\{pt\}$, we have $\mm(\lambda+ n\mu)=
[\Sym(\mathbb{W}^*)\otimes \C_{-n\mu}\otimes (V_\lambda^{G_\mu})^*\vert_H]^H$ 
if $n$ is large enough, for any dominant weight $\lambda$.

\medskip

The last thing that we need to prove is the following

\begin{prop}
$\bullet\quad \mm(n\mu)=0, n\geq 1 \ \Longleftrightarrow \ M_\mu=\emptyset$,

$\bullet\quad \mm(n\mu)$ is non-zero and bounded $\Longleftrightarrow \ M_\mu=\{pt\}$.
\end{prop}

\begin{proof}We will show that the proof follows from Lemma \ref{lem-de-base}.

The symplectic manifold $M=G\times_K(\qgot\oplus V)$ admits a natural identification with the complex manifold $G_\C\times_{K_\C} V$, through the map $[g; X\oplus v]\mapsto [ge^{iX}; v]$. Hence $M$ inherits a $G_\C$-action and a $G_\C$-invariant (integrable) complex structure $J_M$: it is not difficult to check that $J_M$ is compatible with the symplectic form $\Omega_M$.

We are in the setting of Section \ref{sec:Kahler}, where the trivial line bundle $\C\to M$ prequantizes $\Omega_M$. In this context, the space 
$\Gamma(M,\C^{\otimes n})$ of holomorphic section does not depends on $n\in \mathbb{N}$ and is equal to the vector space 
$\mathcal{C}^{hol}(M)$ of holomorphic functions on $M$.

Let us recall Lemma \ref{lem-de-base} which compares the behaviour of the multiplicities
$\mm^{hol}(\mu):= \dim \left[\mathcal{C}^{hol}(M)\otimes (V^G_{\mu})^*\right]^G$ with the reduced spaces $M_\mu:=\Phi^{-1}_M(G\mu)/G$.

\begin{lem}
$\bullet$  $\mm^{hol}(n\mu)=0,\ n\geq 1 \ \Longleftrightarrow \ M_\mu=\emptyset$.

$\bullet$ $\mm^{hol}(n\mu)$ is non-zero and bounded $\Longleftrightarrow \ M_\mu=\{pt\}$. 
\end{lem}

Since the vector space $\mathcal{C}^{hol}(G_\C\times_{K_\C} V)$ admits the vector space 
$$
\bigoplus_{\lambda\in \hat{G}}\ V_\lambda^G\otimes \left[(V_\lambda^G)^*\vert_K\otimes \Sym(V^*)\right]^K
$$
as a dense subspace, we know that the multiplicities $\mm^{hol}(\mu)$ and $\mm(\mu)$ coincide. The proof is then completed.
\end{proof}

\section{Examples}\label{section:example}

Let $\rho: G\to\tilde{G}$ be a morphism between two connected compact Lie groups. We denote $d\rho : \ggot \to \tilde{\ggot}$ the induced Lie algebras morphism, and $\pi : \tilde{\ggot}^* \to \ggot^*$ the dual map.

Select maximal tori $T$ in $G$ and $\tilde{T}$ in $\tilde{G}$, such that $\rho(T)\subset \tilde{T}$. We still denote 
$d\rho : \tgot \to \tilde{\tgot}$ the induced map, and $\pi : \tilde{\tgot}^* \to \tgot^*$ the dual map. 
Let $\tilde{\Lambda}\subset\tilde{\tgot}^*$, $\Lambda\subset\tgot^*$ be the set of weights for the torus $\tilde{T}$ and 
$T$: we have naturally that $\pi(\tilde{\Lambda})\subset \Lambda$. 

Let $\tilde{\Rgot}:=\Rgot(\tilde{G},\tilde{T})$ (resp. $\Rgot:=\Rgot(G,T)$) be any set of roots 
for the group $\tilde{G}$ (resp. $G$). Recall that an element $\tilde{\xi}\in\tilde{\tgot}^*$ defines a parabolic sub-algebra $\tilde{\pgot}_{\tilde{\xi}}:= \tilde{\tgot}_\C\oplus \sum_{(\alpha,\tilde{\xi})\geq 0}(\tilde{\ggot}_\C)_\alpha$ of the reductive Lie algebra $\tilde{\ggot}_\C$. Its nilpotent radical is 
$\tilde{\ngot}_{\tilde{\xi}}:=  \sum_{(\alpha,\tilde{\xi})> 0}(\tilde{\ggot}_\C)_\alpha$.

\begin{defi}\label{def:lambda-adapted}
An element $\tilde{\xi}\in\tilde{\tgot}^*$ is adapted to the group $G$ if the set 
$\pi(\{\alpha\in\tilde{\Rgot},(\alpha,\tilde{\xi})>0\})$ is contained in an open half space, i.e. if 
there exists $\xi_o\in\tgot^*$ such that $\forall \alpha\in \tilde{\Rgot},\  (\alpha,\tilde{\xi})>0 \Longrightarrow 
(\pi(\alpha),\xi_o) >0$.
\end{defi}

%
%
%
%

\medskip

Let $\tilde{\Ocal}$ be a coadjoint orbit of the group $\tilde{G}$. 
The moment map $\tilde{\Ocal}\to \ggot^*$ relative to the action of $G$ on 
$\tilde{\Ocal}$ is the restriction of $\pi$ on $\tilde{\Ocal}$. Hence for any 
$\xi\in\kgot^*$, the $G$-reduction of $\tilde{\Ocal}$ at $\xi$ is equal to $\tilde{\Ocal}\cap \pi^{-1}(G\xi)/G$.

The main tool used in this section is the following 

\begin{prop} \label{prop:xi-adapted}
Let  $\tilde{\xi}\in\tilde{\tgot}^*$ and $\xi=\pi(\tilde{\xi})$. If $\tilde{\xi}$ is $G$-adapted, we have 
\begin{itemize}
\item the $G$-reduction of the coadjoint orbit $\tilde{G}\tilde{\xi}$ at $\xi$ is reduced to a point,
\item $\rho(G_\xi)\subset \tilde{G}_{\tilde{\xi}}$,
\item $\rho(\pgot_{\xi})\subset \tilde{\pgot}_{\tilde{\xi}}$, where $\pgot_{\xi}\subset\ggot_\C$ and $\tilde{\pgot}_{\tilde{\xi}}\subset\tilde{\ggot}_\C$ are the parabolic sub-algebras defined  respectively 
by $\xi\in\tgot^*$ and $\tilde{\xi}\in\tilde{\tgot}^*$,
\item The linear map $\rho : \pgot_{\xi}\to \tilde{\pgot}_{\tilde{\xi}}$ factorizes to a linear map 
$\overline{\rho}: \ngot_{\xi}\to \tilde{\ngot}_{\tilde{\xi}}$.
\end{itemize}
\end{prop}

\begin{proof} It is immediate to see that the first two points are a consequence of the following equality
\begin{equation}\label{eq:S}
\tilde{G}\tilde{\xi}\cap \pi^{-1}(\xi)=\{\tilde{\xi}\}. 
\end{equation}

Let us denote $\pi_{\tilde{\tgot}}:\tilde{\ggot}^*\to \tilde{\tgot}^*$ the projection. Since 
$\tilde{G}\tilde{\xi}\cap \pi_{\tilde{\tgot}}^{-1}(\tilde{\xi})$ is reduced to the singleton $\{\tilde{\xi}\}$, 
the identity (\ref{eq:S}) follows from the following identity
\begin{equation}\label{eq:S-prime}
\pi_{\tilde{\tgot}}(\tilde{\Ocal})\cap \pi_{\tilde{\tgot}}\left(\pi^{-1}(\xi)\right)=\{\tilde{\xi}\}. 
\end{equation}
Thanks to the Convexity Theorem \cite{LMTW} we know that $\pi_{\tilde{\tgot}}(\tilde{\Ocal})$ is equal 
to the convex hull $\mathrm{Conv}(\tilde{W}\tilde{\xi})$, where $\tilde{W}$ is the Weyl group of 
$(\tilde{G},\tilde{T})$. On the other hand the set $\pi_{\tilde{\tgot}}\left(\pi^{-1}(\xi)\right)$ 
is equal to the affine subspace $\tilde{\xi}+ E$ where $E\subset\tilde{\tgot}^*$ is equal to the kernel of 
$\pi : \tilde{\tgot}^* \to \tgot^*$. Let $\Acal\subset \tilde{\tgot}^*$ be the tangent cone at $\tilde{\xi}$ 
of the convex set $\mathrm{Conv}(\tilde{W}\tilde{\xi})$: by standard computation we know that 
$-\Acal$ is the cone generated by $\alpha\in\tilde{\Rgot},(\alpha,\tilde{\xi})>0$. 
Since $\pi_{\tilde{\tgot}}(\tilde{\Ocal})\subset \tilde{\xi}+ \Acal$ 
we see that (\ref{eq:S-prime}) is a consequence of 
\begin{equation}\label{eq:E}
\Acal \cap E=\{0\}. 
\end{equation}
Our proof of (\ref{eq:S}) is now completed since (\ref{eq:E}) follows immediately from the fact that for some 
$\xi_o\in\tgot$ we have: $\forall \alpha\in \tilde{\Rgot},\  (\alpha,\tilde{\xi})>0 \Longrightarrow 
(\pi(\alpha),\xi_o) >0$.

Let us concentrate to the third point. We know already that $\rho(G_\xi)\subset \tilde{G}_{\tilde{\xi}}$. Hence to get the inclusion 
$\rho(\pgot_{\xi})\subset \tilde{\pgot}_{\tilde{\xi}}$ we have just to check that 
\begin{equation}\label{eq:beta-1}
\rho((\ggot_\C)_\beta)\subset \tilde{\pgot}_{\tilde{\xi}}
\end{equation} 
for any $\beta \in\Rgot$ such that $(\beta,\xi)>0$. A small computation shows that (\ref{eq:beta-1}) is a consequence of 
\begin{equation}\label{eq:beta-2}
\left\{\alpha\in\tilde{\Rgot}, (\alpha,\tilde{\xi})<0\right\}\bigcap \pi^{-1}(\beta)=\emptyset.
\end{equation}
It is proved in \cite{Heckman82}[Lemma 8.3], that
\begin{equation}\label{eq:beta-3}
\left\{\beta\in\Rgot, (\beta, \xi)>0\right\}\subset 
\pi\left(\left\{\alpha\in\tilde{\Rgot}, (\alpha,\tilde{\xi})>0\right\}\right).
\end{equation}
Since $\tilde{\xi}\in\tilde{\tgot}^*$ is adapted to the group $G$, we have that
\begin{equation}\label{eq:beta-4}
\pi\left(\left\{\alpha\in\tilde{\Rgot}, (\alpha,\tilde{\xi})>0\right\}\right)\bigcap
\pi\left(\left\{\alpha\in\tilde{\Rgot}, (\alpha,\tilde{\xi})<0\right\}\right)=\emptyset.
\end{equation}
Hence (\ref{eq:beta-2}) follows from the identities  (\ref{eq:beta-3}) and (\ref{eq:beta-4}).

For the last point we just use that the linear map $\rho : \pgot_{\xi}\to \tilde{\pgot}_{\tilde{\xi}}$ sends 
$(\ggot_\xi)_\C$ into $(\tilde{\ggot}_{\tilde{\xi}})_\C$. Then it factorizes to a map 
$\overline{\rho}$ from $\ngot_{\xi}\simeq \pgot_{\xi}/(\ggot_{\xi})_\C$ into 
$\tilde{\ngot}_{\tilde{\xi}}\simeq \tilde{\pgot}_{\tilde{\xi}}/(\tilde{\ggot}_{\tilde{\xi}})_\C$.

\end{proof}

\medskip

Let us fix some set of dominant weights $\tilde{\Lambda}_{\geq 0}$, $\Lambda_{\geq 0}$ for the groups 
$\tilde{G}$ and $G$. For any $(\mu,\tilde{\mu})\in \Lambda_{\geq 0}\times\tilde{\Lambda}_{\geq 0}$,
we denote $V^G_\mu$, $V^{\tilde{G}}_{\tilde{\mu}}$ the corresponding irreducible representations of $G$ and $\tilde{G}$, 
and we define $\mm(\mu,\tilde{\mu})$ as the multiplicity of $V_{\mu}^G$ in $V^{\tilde{G}}_{\tilde{\mu}}\vert_{G}$.

We give now a specialization of Theorem B.

\begin{theo}\label{theoreme-B-second}
Let $(\tilde{\mu},\tilde{w})\in\tilde{\Lambda}_{\geq 0}\times\tilde{W}$ such that $\tilde{w}\tilde{\mu}$ is adapted to $G$. 

Up to the conjugation by an element of the Weyl group of $G$ we can assume that $\mu_{\tilde{w}}:=
\pi(\tilde{w}\tilde{\mu})$ is a dominant weight. We denote $H\subset G$ and $\tilde{H}\subset \tilde{G}$ the respective 
stabilizers\footnote{Recall that $\rho(H)\subset\tilde{H}$.} of $\mu_{\tilde{w}}$ and 
$\tilde{w}\tilde{\mu}$.

\begin{itemize}
\item We have $\mm(n\mu_{\tilde{w}},n\tilde{\mu})=1$, for all $n\geq 1$.

\item For any dominant weight $(\lambda,\tilde{\lambda})$ the sequence 
$\mm(\lambda+ n\mu_{\tilde{w}},\tilde{\lambda}+ n\tilde{\mu})$ is increasing 
and equal to 
$$
\dim \left[\Sym(\mathbb{W}^*)\otimes V_{\tilde{\lambda}}^{\tilde{H}}\vert_H\otimes (V_{\lambda}^H)^*\right]^H
$$
for $n$ large enough. Here $\mathbb{W}$ corresponds to the $H$-module 
\begin{equation}\label{eq:W-particulier}
\tilde{\ngot}_{\tilde{w}\tilde{\mu}}/\overline{\rho}(\ngot_{\mu_{\tilde{w}}}).
\end{equation}
\end{itemize}
\end{theo}

\begin{proof}
The first point is due to the fact that the stabilizer of $\tilde{w}\tilde{\mu}$ relative to the $G$-action is 
equal to the connected subgroup $H$, hence the $H$-module $\mathbb{D}$ is trivial. For the second point we have just to check the computation of the $H$-module $\mathbb{W}$. Let $a=\tilde{w}\tilde{\mu}\in \tilde{\Ocal}:=\tilde{G}\tilde{\mu}$.
Here $\T_{a}\tilde{\Ocal}\simeq \tilde{\pgot}_{\tilde{w}\tilde{\mu}}/ \tilde{\hgot}_\C$. As $\rho(\pgot_{\mu_{\tilde{w}}})\subset \tilde{\pgot}_{\tilde{w}\tilde{\mu}}$
one sees directly that $\mathbb{W}=\T_{a}\tilde{\Ocal}/\rho(\pgot_{\mu_{\tilde{w}}})\cdot a$ is equal to (\ref{eq:W-particulier}).
\end{proof}

\bigskip

We have another specialization of Theorem B that will be used in the plethysm case. We suppose here that the sets 
of positive roots $\tilde{\Rgot}^+$ and $\Rgot^+$ are chosen so that the corresponding Borel subgroups 
$B\subset G_\C$ and $\tilde{B}\subset \tilde{G}_\C$ satisfy 
\begin{equation}\label{eq:B-tilde-B}
\rho(B)\subset \tilde{B}.
\end{equation} 
Let $\tilde{\Lambda}_{\geq 0}$, $\Lambda_{\geq 0}$ be the corresponding set of dominants weight. 
When we work with this parametrization we have the following classical fact.
\begin{lem}
Let $\tilde{\mu}\in \tilde{\Lambda}_{\geq 0}$ and $\mu=\pi(\tilde{\mu})$. We have 
\begin{enumerate}
\item $\mu\in \Lambda_{\geq 0}$ and $\mm(\mu,\tilde{\mu})\neq 0$,
\item $\rho(\pgot_\mu)\subset \tilde{\pgot}_{\tilde{\mu}}$ and $\rho(G_\mu)\subset \tilde{G}_{\tilde{\mu}}$.
\end{enumerate}
\end{lem}

\begin{proof}
Let $\tilde{V}_{\tilde{\mu}}$ be an irreducible representation of $\tilde{G}$ with highest weight 
$\tilde{\mu}$. There exists a non-zero vector $v_o\in \tilde{V}_{\tilde{\mu}}$ such that the line $\C v_o$ is fixed 
by $\tilde{B}$ and the maximal torus $\tilde{T}$ acts on $\C v_o$ through the character $\tilde{t}\mapsto \tilde{t}^{\tilde{\mu}}$. 

Let $V$ be the vector space generated by $\rho(g)v_o, g\in G$. It is an irreducible representation of $G$ and 
$v_o$ is still a highest weight vector for the $G$-action~:  the line $\C v_o$ is fixed by $B$ and the maximal 
torus $T$ acts on $\C v_o$ through the character $t\mapsto t^{\mu}$. This forces $\mu$ to be a dominant weight for 
$G$ (relatively to $B$) and then $V\subset \tilde{V}_{\tilde{\mu}}$ is $G$-representation with highest weight $\mu$ : the first point is proved.

For the second point we look at the $\tilde{G}_\C$-action (resp. $G_\C$-action ) on the projective space 
$\mathbb{P}(\tilde{V}_{\tilde{\mu}})$ (resp. $\mathbb{P}(V)$), the stabilizer subgroup of the line $\C v_o$ 
is equal to the parabolic subgroup $\tilde{P}_{\tilde{\mu}}\subset \tilde{G}_\C$ (resp. $P_\mu\subset G_\C$) : 
hence $\rho(P_\mu)\subset \tilde{P}_{\tilde{\mu}}$. If we work with the actions of the compact groups $G$ and 
$\tilde{G}$ we get similarly that $\rho(G_\mu)\subset \tilde{G}_{\tilde{\mu}}$.
\end{proof}

\medskip

Like in Proposition \ref{prop:xi-adapted}, the linear map $\rho : \pgot_{\mu}\to \tilde{\pgot}_{\tilde{\mu}}$ 
factorizes to a linear map $\overline{\rho}: \ngot_{\mu}\to \tilde{\ngot}_{\tilde{\mu}}$. We have another specialization of Theorem B.

\begin{theo}\label{theoreme-B-second}
Suppose that (\ref{eq:B-tilde-B}) holds. Let $\tilde{\mu}\in\tilde{\Lambda}_{\geq 0}$ and  $\mu:=\pi(\tilde{\mu})\in\Lambda_{\geq 0}$. 
We denote $H\subset G$ and $\tilde{H}\subset \tilde{G}$ the respective stabilizers\footnote{Recall that $\rho(H)\subset\tilde{H}$.} 
of $\mu$ and $\tilde{\mu}$. Let $\mathbb{W}:=\tilde{\ngot}_{\tilde{\mu}}/\overline{\rho}(\ngot_{\mu})$.

The following statements are equivalent:
\begin{itemize}
\item[a)] $\mm(n\mu,n\tilde{\mu})=1$, for all $n\geq 1$.

\item[b)] For any dominant weight $(\lambda,\tilde{\lambda})$ the increasing sequence 
$\mm(\lambda+ n\mu,\tilde{\lambda}+ n\tilde{\mu})$ has a limit.

\item[c)] The algebra $\Sym(\mathbb{W}^*)$ has finite $H$-multiplicities.
\end{itemize}

If these statements hold the limit of the sequence $\mm(\lambda+ n\mu,\tilde{\lambda}+ n\tilde{\mu})$ is equal to 
the multiplicity of $V_{\lambda}^H$ in the $H$-module $\Sym(\mathbb{W}^*)\otimes V_{\tilde{\lambda}}^{\tilde{H}}$.
\end{theo}

\begin{proof}
We have constructed $(\mu,\tilde{\mu})$ so that $\mm(\mu,\tilde{\mu})\neq 0$. In this case proposition \ref{prop:theorem-B} and Theorem B tells us that the following equivalences hold 
$\mm(n\mu,n\tilde{\mu})=1, \forall n\geq 1\Longleftrightarrow \mm(n\mu,n\tilde{\mu})$ 
is bounded $\Longleftrightarrow (\tilde{G}\tilde{\mu})_\mu=\{pt\}$. Hence we have proved that $a)\Leftrightarrow c)$ and 
$b)\Rightarrow a)$. The other implication $a)\Rightarrow b)$ is also a consequence of Theorem B.
\end{proof}

\subsection{The Littlewood-Richardson coefficients}

Here we work with $G$ embedded diagonally in $\tilde{G}:=G\times G$. The map 
$\pi:\ggot^*\times \ggot^*\to \ggot^*$ is defined by $(\xi_1,\xi_2)\mapsto \xi_1+\xi_2$.

Here the multiplicity function 
$\mm: \Lambda^+_{\geq 0}\times \Lambda^+_{\geq 0}\times \Lambda^+_{\geq 0}\to \mathbb{N}$ is defined by 
$$
\mm(a,b,c):= \dim \left[ (V_a^G)^*\otimes V_b^G\otimes V_c^G\right]^G
$$

We fix an element $(\mu_1,\mu_2)\in(\Lambda^+_{\geq 0})^2$. It is easy to see that $(\mu_1,\mu_2)$ is adapted to $G$. We denote $\mu=\mu_1+\mu_2$. The stabilizer subgroup $G_{\mu}$ is equal to $G_{\mu_1}\cap G_{\mu_2}$. We work with the 
$G_{\mu}$-module 
\begin{equation}\label{eq:W-littlewood}
\mathbb{W}_{\mu_1,\mu_2}:=\sum_{\stackrel{(\alpha,\mu_1)>0}{(\alpha,\mu_2)>0}}(\ggot_\C)_\alpha.
\end{equation}

In this case Theorem \ref{theoreme-B-second} gives 

\begin{prop} Let $(\mu_1,\mu_2)\in(\Lambda^+_{\geq 0})^2$ and $\mu=\mu_1+\mu_2$.
\begin{itemize}
\item We have $\mm(n\mu,n\mu_1,n\mu_2)=1$ for any $n\geq 1$.

\item For any $(a,b,c)\in (\Lambda^+_{\geq 0})^3$, the sequence $\mm(a+n\mu,b+n\mu_1,c+n\mu_2)$ is increasing and 
equal to 
$$
\dim \left[\Sym(\mathbb{W}_{\mu_1,\mu_2}^*)\otimes (V^{G_{\mu}}_a)^*\otimes V^{G_{\mu_1}}_b\vert_{G_{\mu}}\otimes V^{G_{\mu_2}}_c\vert_{G_{\mu}}\right]^{G_{\mu}}.
$$
for $n$ large enough.
\end{itemize}
\end{prop}

\begin{proof}
If we follows the notation of Theorem \ref{theoreme-B-second}, we have $\tilde{\mu}=(\mu_1,\mu_2)$, 
$\tilde{w}=1$, $\mu_{\tilde{w}}=\mu=\mu_1+\mu_2$, the parabolic subgroups 
$\tilde{\pgot}_{\tilde{w}\tilde{\mu}}, \pgot_{\mu_{\tilde{w}}}$ are respectively equal to 
$\pgot_{\mu_1}\times \pgot_{\mu_2}$ and $\pgot_{\mu_1}\cap \pgot_{\mu_2}$ and the subgroup $\tilde{H}$ is equal to 
$G_{\mu_1}\times G_{\mu_2}$. We check then easily that the $G_\mu$-module 
$\tilde{\ngot}_{\tilde{w}\tilde{\mu}}/\overline{\rho}(\ngot_{\mu_{\tilde{w}}})$ 
is equal to $\mathbb{W}_{\mu_1,\mu_2}$.
\end{proof}

\subsection{The Kronecker coefficients}

Let $\U(E),\U(F)$ be the unitary groups of two hermitian vector spaces $E,F$. The aim of this section 
is to detail our results for the canonical morphism 
$$
\rho:G:=\U(E)\times\U(F)\to\tilde{G}:=\U(E\otimes F).
$$
This problem is equivalent to the  question on the decomposition of tensor products of 
representations for the symmetric group.

A partition $\lambda$ is a sequence $\lambda = (\lambda _1,\lambda_2,\ldots,\lambda_k)$ of weakly 
decreasing non-negative integers. By convention, we allow partitions with some zero parts, and two 
partitions that differ by zero parts are the same. For any partition $\lambda$, we define 
$\vert\lambda\vert = \lambda_1 +\lambda_2 +\cdots+ \lambda_k$ and $l(\lambda)$ 
as the number of non-zero parts of $\lambda$.

Recall that the the $\U(E)$ irreducible polynomial representations are in bijection with the partitions 
$\lambda$ such that $l(\lambda)\leq \dim E$.  We denote by $S_\lambda(E)$ the representation 
associated to $\lambda$.

We consider the groups $G:=\U(E)\times\U(F)$ and $\tilde{G}:=\U(E\otimes F)$. 
Let $\gamma$ be a partition such that $l(\gamma)\leq \dim E\cdot \dim F$. We can decompose 
the irreducible representation $S_\gamma(E\otimes F)$ as a $G$-representation:
$$
S_\gamma(E\otimes F)=\sum_{\alpha,\beta} \ g(\alpha,\beta,\gamma) \ S_\alpha(E)\otimes S_\gamma(F)
$$
where the sum is taken over partitions $\alpha$, $\beta$ such that 
$\vert\alpha\vert = \vert\beta\vert = \vert\gamma\vert$, $l(\alpha)\leq \dim E$ and $l(\beta)\leq \dim F$.

We fix an orthonormal basis $(e_i)$ for $E$, $(f_j)$ for $F$: let $(e_i\otimes f_j)$ the corresponding orthonormal 
basis of $E\otimes F$. We denote $T_E$ (resp. $T_F$) the maximal tori of $\U(E)$ (resp. $\U(F)$) consisting of 
endomorphism that are diagonal over $(e_i)$ (resp. $(f_j)$). We denote $T=T_E\times T_F$ the maximal torus of $G$. Similarly we denote $\tilde{T}$ the maximal tori of $\tilde{G}$ associated to the endomorphisms that diagonalize the basis $(e_i\otimes f_j)$. At the level of tori, the morphism $\rho$ induces a map
$\rho : T\to \tilde{T}$ sending $((t_i), (s_j))$ to 
$(t_is_j)$. At the level of Lie algebra the map $\rho:\tgot\to\tilde{\tgot}$ is defined by 
$$
\rho(x, y)= (x_i+y_j)_{i,j}
$$
for $x=(x_1,\cdots,x_{\dim E})\in \R^{\dim E}\simeq \mathrm{Lie}(T_E)$ and $y=(y_1,\cdots,y_{\dim F})\in \R^{\dim F}\simeq \mathrm{Lie}(T_F)$. 

Let $\theta_{kl}\in\tilde{\tgot}^*$ be the linear form that send an element $(a_{i,j})\in\tilde{\tgot}$ to $a_{kl}$. 
Then $\tilde{\tgot}^*$ is canonically identified with the vector space of matrices of size $\dim E\times \dim F$ through 
the use of the basis $\theta_{kl}$, and the dual map $\pi :\tilde{\tgot}^*\to \tgot^*$ is given by $\pi((\xi_{ij}))= ((\sum_j \xi_{ij})_i,(\sum_i \xi_{ij})_j)$.

Recall the following definition \cite{Vallejo14,Manivel14}. 

\begin{defi} 
Let $A = (a_{i,j})$ be a matrix of size $\dim E\times \dim F$. Then, A is called {\em additive} if there exist 
real numbers $x_1,\ldots,x_{\dim E}$, $y_1,\ldots,y_{\dim F}$ such that
$$
a_{i,j} > a_{k,l} \Longrightarrow x_i + y_j > x_k + y_l,
$$
for all $i,k\in [1,\ldots, \dim E]$ and all $j, l \in [1,\ldots,\dim F]$. 
\end{defi}

The following easy fact is important.

\begin{lem}
Let $\xi\in\tilde{\tgot}^*$ that is represented by a matrix $(\xi_{ij})$. Then 
$\xi$ is adapted to the group $G$ if and only if the matrix $(\xi_{ij})$ is additive.
\end{lem}
\begin{proof} The system of roots for $\tilde{G}$ is $\tilde{\Rgot}=\{\theta_{ij}-\theta_{kl}, (i,j)\neq (k,l)\}$.
By definition $\xi\in\tilde{\tgot}^*$ is adapted to $G$ if and only if there exists 
$(x,y)\in\R^{\dim E}\times \R^{\dim F}\simeq \tgot^*$ such that 
$$
(\theta_{ij}-\theta_{kl},\xi)> 0\Longrightarrow (\pi(\theta_{ij}-\theta_{kl}),(x,y)).
$$
Our proof is completed since $(\theta_{ij}-\theta_{kl},\xi)=\xi_{ij}-\xi_{kl}$ and 
$(\pi(\theta_{ij}-\theta_{kl}),(x,y))=x_i + y_j - (x_k + y_l)$.
\end{proof}

\begin{defi}
If $A = (a_{i,j})$ is a matrix of size $\dim E\times \dim F$ with non negative integral coefficients, we define the 
partition $\alpha_A,\beta_A,\gamma_A$ where $\alpha_A\simeq (\sum_j a_{ij})_i$, $\beta_A\simeq 
(\sum_i a_{ij})_j$ and $\gamma_A\simeq (a_{i,j})$. Note that $|\alpha_A|=|\beta_A|=|\gamma_A|$.
\end{defi}

\medskip

The first part of Theorem \ref{theoreme-B-second} permits us to recover the following result of 
Vallejo \cite{Vallejo14} and Manivel \cite{Manivel14}.

\begin{prop}
Let $A = (a_{i,j})$ is a matrix of size $\dim E\times \dim F$ with non negative integral coefficients. 
If the matrix $A$ is {\em additive} then 

$\bullet$ $g(n\alpha_A,n\beta_A,n\gamma_A)=1$ for all $n\geq 1$,

$\bullet$ the sequence $g(a+n\alpha_A,b+n\beta_A,c+n\gamma_A)$ is increasing and stationary 
for any partition $a,b,c$ such that $|a|=|b|=|c|$, $l(a)\leq\dim E$, $l(b)\leq \dim F$ and 
$l(c)\leq \dim E\cdot \dim F$.
\end{prop}

\medskip

Now we want to exploit the second part of Theorem \ref{theoreme-B-second} that concerns a formula for the limit multiplicities. 

\begin{defi} Let $A = (a_{i,j})$ is an {\em additive} matrix of size  $\dim E\times \dim F$ 
with non negative integral coefficients. For any partition $a,b,c$ such that $|a|=|b|=|c|$, we define 
$g_A(a,b,c)\in \mathbb{N}$ as the limit of the sequence $g(a+n\alpha_A,b+n\beta_A,c+n\gamma_A)$ when $n\to\infty$.
\end{defi}
%
%
%
%
%
%

We denote $E_i^k$ (resp. $F_j^l$) the orthogonal projection of rank $1$ of $E$ (resp. $F$) that sends 
$e_i$  to $e_k$ (resp. $f_j$ to $f_l$). At an additive matrix $A$, we attach :

$\bullet$ The stabilizer $\tilde{H}_A\subset \tilde{G}$ of the element $A\in\tilde{\tgot^*}$, with Lie algebra $\tilde{\hgot}_A$.

$\bullet$ The stabilizer $H_A\subset G$ of the element $\pi(A)$. We have 
$H_A=H_A^E\times H_A^F$ with $H_A^E=\U(E)_{\alpha_A}$ and $H_A^F=\U(F)_{\beta_A}$.

$\bullet$ The $\tilde{H}_A$-module 
$$
\tilde{\pgot}_A:=\sum_{a_{ij}\geq a_{kl}} \C \, E_i^k\otimes F_j^l
$$
that corresponds to the parabolic  sub-algebra of $\tilde{\ggot}_\C$ attached to $A$. Its nilradical is 
$\tilde{\ngot}_A=\sum_{a_{ij}>a_{kl}} \C \, E_i^k\otimes F_j^l$.

$\bullet$ the sub-algebras $\ngot_{\pi(A)}\subset\pgot_{\pi(A)}\subset \ggot_\C$ and their images by $\rho$:
\begin{eqnarray*}
\rho(\pgot_{\pi(A)})&=&\sum_{\alpha_{i}\geq \alpha_k} 
\C\,E_{i}^{k}\otimes \mathrm{Id}_F\oplus \sum_{\beta_{j}\geq \beta_l}  \C\,\mathrm{Id}_E\otimes F_{j}^{l}\\
\rho(\ngot_{\pi(A)})&=&\sum_{\alpha_{i}> \alpha_k} 
\C\,E_{i}^{k}\otimes \mathrm{Id}_F\oplus \sum_{\beta_{j} >\beta_l}  \C\,\mathrm{Id}_E\otimes F_{j}^{l}
\end{eqnarray*}

Thanks to proposition \ref{prop:xi-adapted} we know that $\rho(H_A)\subset \tilde{H}_A$ 
and that $\rho(\pgot_{\pi(A)})\subset\tilde{\pgot}_A$. We denote 
$\overline{\rho}(\ngot_{\pi(A)})$ the projection of $\rho(\ngot_{\pi(A)})\subset\tilde{\pgot}_A$ on 
$\tilde{\pgot}_A/(\tilde{\hgot}_A)_\C\simeq\tilde{\ngot}_A$.

We define the $H_A$-module 
\begin{equation}\label{eq:W-A}
\mathbb{W}_A=\tilde{\ngot}_A/\overline{\rho}(\ngot_{\pi(A)})
\end{equation}
and we know that $\Sym(\mathbb{W}_A^*)$ has finite $H_A$-multiplicities.

For a partition $a=(a_1,a_2,\ldots,a_{\dim E})$, we define $V^{H^E_A}_a$ as the irreducible representation of $H_A^E$ 
with highest weight $a$. If $\alpha_A=(l_1^{n_1},l_2^{n_2},\ldots, l_r^{n_r})$ with $l_1>l_2>\cdots>l_r$, the 
subgroup $H^E_A$ is isomorphic to $\U(E_1)\times\cdots \times \U(E_r)$ with $\dim E_k= n_k$, and the 
representation $V^{H^E_A}_a$ is equal to the tensor product 
$S_{a[1]}(E_1)\otimes S_{a[2]}(E_r)\otimes\cdots \otimes S_{a[r]}(E_r)$ where $a[k]$ is the partition 
$(a_{n_1+\ldots+ n_r+1},\ldots, a_{n_1+\ldots+ n_{r+1}})$. 

We can define similarly the representations $V^{\tilde{H}_A}_c$ and $V^{H^F_A}_b$. Theorem \ref{theoreme-B-second} give us the following

\begin{theo}\label{theo-stretched}
Let $A = (a_{i,j})$ is a {\em additive} matrix of size $\dim E\times \dim F$ with non negative integral coefficients.
For any partition $a,b,c$ such that $|a|=|b|=|c|$, $l(a)\leq\dim E$, $l(b)\leq \dim F$ and 
$l(c)\leq \dim E\cdot \dim F$, we have
$$
g_A(a,b,c)= \dim\left[\Sym(\mathbb{W}_A^*)\otimes 
(V^{H^E_A}_a)^*\otimes (V^{H^F_A}_b)^*\otimes V^{\tilde{H}_A}_c\vert_{H_A^E\times H_A^F}\right]^{H_A^E\times H_A^F}
$$
\end{theo}

\medskip

\subsubsection{The partition $(1^{pq})$}

Let us work out the example of the partition $A=(1^{pq})$ where $1\leq p \leq \dim E$ and $1\leq q\leq\dim F$. 

We see $A=(1^{pq})$ as an additive matrix $(a_{ij})$ of type $\dim E\times \dim F$: 
$a_{ij}$ is non-zero, equal to $1$, only if $1\leq i \leq p$ and $1\leq j\leq q$. We denote $g_{pq}$ 
the corresponding stretched Kronecker coefficients. 

We use an orthogonal decomposition of our vector spaces : $E=E_p\oplus E'$ and $F=F_q\oplus F'$ with 
$\dim E_p=p$ and $\dim F_q=q$. For the tensor product we have 
$E\otimes F= E_p\otimes F_q\oplus (E_p\otimes F_q)^\perp$ where $(E_p\otimes F_q)^\perp= 
E_p\otimes F' \oplus E'\otimes F_q\oplus E'\otimes F'$. 

The stabiliser subgroup of $A$ in $\tilde{G}$ is $\tilde{H}_{pq}:= 
\U(E_p\otimes F_q)\times  \U((E_p\otimes F_q)^\perp)$ and 
the stabiliser subgroup of $\pi(A)$ in $G$ is $H_{pq}:= H^E_p\times H^F_q$ where 
$H^E_p= \U(E_p)\times \U(E')$ and $H^F_q= \U(F_q)\times \U(F')$.

If $A=(1^{pq})$, we denote $\mathbb{W}_A=\mathbb{W}_{pq}$ the $H_{pq}$-module introduced in  (\ref{eq:W-A}). 
A direct computation shows that
\begin{eqnarray*}
\lefteqn{\mathbb{W}_{pq}=\hom(E_p,E')\otimes \slgot(F_q) \bigoplus}\\
&&  \slgot(F_q)\otimes \hom(F_q,F')\bigoplus \hom(E_p,E')\otimes \hom(F_q,F').
\end{eqnarray*}

A partition $a=(a_1,\ldots,a_{\dim E})$ defines the partitions $a(p):=(a_1,\ldots,a_{p})$ and 
$a':=(a_{p+1},\ldots,a_{\dim E})$. Similarly a partition $b=(b_1,\ldots,b_{\dim F})$ defines 
the partitions $b(q):=(b_1,\ldots,b_{q})$ and $b':=(a_{q+1},\ldots,a_{\dim F})$.

A partition $c$ of length $\dim E\times \dim F$ is represented by a matrix $(c_{ij})$. We define then the partition
$c(pq)$ of length $pq$ represented by the coefficients $c_{ij}$ when $1\leq i\leq p$ and $1\leq j\leq q$, and the partition 
$c'$ which is the complement of $c(pq)$ in $c$. 

Theorem \ref{theo-stretched} tell us that the stretched Kronecker coefficient $g_{pq}(a,b,c)$ is equal to the multiplicity of 
the irreducible representation 
$$
S_{a(p)}(E_p)\otimes S_{a'}(E')\otimes S_{b(q)}(F_q)\otimes S_{b'}(F')
$$
in 
$$
\Sym(\mathbb{W}_{pq}^*)\otimes S_{c(pq)}(E_p\otimes F_q)\otimes S_{c'}((E_p\otimes F_q)^\perp).
$$

When $q=1$ the following expression for the stretched coefficient was obtained by Manivel 
\cite{Manivel14}, extending the case $p=q=1$ treated by Brion \cite{Brion93}.

\subsubsection{The triple $(22), (22), (22)$}

The aim of this section is to explain how our technique permit us to recover the result of 
Stembridge \cite{Stembridge14} concerning the stability of the triple $(22),(22),(22)$. 

We work with the morphism $\rho :\U(\C^2)\times \U(\C^2)\to \U(\C^2\otimes\C^2)$. 
The matrix
$$
\tilde{\lambda}:=i\left(\begin{array}{cc} 
1 & 0 \\ 
0 & 1 
\end{array}\right).
$$
represents a weight of the maximal torus $\tilde{T}$ of $\tilde{G}=\U(\C^2\otimes\C^2)$. Let 
$\chi_{\tilde{\lambda}}$ be the character defined by $\tilde{\lambda}$ on the stabilizer subgroup of $\tilde{G}_{\tilde{\lambda}}$.

The restriction of $\tilde{\lambda}$ to the maximal torus $T$ of $G=\U(\C^2)\times \U(\C^2)$ defines a weight $\lambda=\pi(\tilde{\lambda})$. We see that $\lambda$ is the differential of the character 
$\chi_\lambda:=\det\times\det$. 

The Kronecker coefficient $g(n(1,1),n(1,1),n(1,1))$ correspond to the multiplicity of the character 
$\chi_\lambda^{\otimes n}$ in $V^{\tilde{G}}_{n\tilde{\lambda}}$. Let us check that the sequence  
$g(n(1,1),n(1,1),n(1,1))$ is bounded.

The subgroup of $G$ that stabilizes $\tilde{\lambda}$ is denoted $H:=G\cap \tilde{G}_{\tilde{\lambda}}$. 
Let $H^o$ be its connected component. We consider the following $H$-module 
$\mathbb{W}:=\T_{\tilde{\lambda}}(\tilde{G}\tilde{\lambda})/\ggot_\C\cdot\tilde{\lambda}$.

\begin{lem}
\begin{enumerate}
\item The $H$-module $\mathbb{W}$ is reduced to $\{0\}$. 
\item The reduced space $(\tilde{G}\tilde{\lambda})_\lambda$ is a singleton.
\item The character $\chi_{\tilde{\lambda}}\chi_\lambda^{-1}$ is trivial on $H^o$ and 
defines an isomorphism between $H/H^o$ and $\{\pm 1\}$.
\item $g(n(1,1),n(1,1),n(1,1))=\frac{1+(-1)^n}{2}$.
\end{enumerate}

\end{lem}

\begin{proof}
If we compute the real dimensions we have $\dim \tilde{G}\tilde{\lambda}=\dim \U(4)- 2\dim \U(2)=8$. On the 
other hand, $\dim \ggot_\C\cdot\tilde{\lambda}=2\dim \ggot\cdot\tilde{\lambda}=2(\dim G-\dim H)$. But 
one can compute easily that $H^o=T$. Hence $\dim H= 4$ and $\dim \ggot_\C\cdot\tilde{\lambda}=\dim \tilde{G}\tilde{\lambda}$. It proves the first point.

The second point is a consequence of the first point (see Proposition \ref{prop:theorem-B}). At this stage we know that 
$$
g(n(1,1),n(1,1),n(1,1))=\dim [(\chi_{\tilde{\lambda}}\chi_\lambda^{-1})^{\otimes n}]^H.
$$
The last point is a consequence of the third one. The easy checking of the third point is left to the reader.
\end{proof}

\subsection{Plethysm}

Let $\rho: G\to \tilde{G}:=\U(V)$ be an irreducible representation of the group $G$. Let $N=\dim V$. 
Let $T$ be a maximal torus of $G$. The $T$-action on $V$ can be diagonalized: there exists 
an orthonormal basis $(v_j)_{j\in J}$  and a family of 
weights $(\alpha_j)_{j\in J}$ such that $\rho(t)v_j= t^{\alpha_j}v_j$ for all $t\in T$. Let $\tilde{T}$ be the 
maximal torus of $\tilde{G}$ formed by the the unitary endomorphisms that are diagonalized by the basis 
$(v_j)_{j\in J}$: we have then $\rho(T)\subset \tilde{T}$. We denote $\pi:\tilde{\tgot}^*\to\tgot^*$ 
the projection, and $e_k\in\tilde{\tgot}^*$ the linear form that sends 
$(x_j)_{j\in J}$ to $x_k$.

Let $B$ be a Borel subgroup of $G$: there exists a Borel subgroup $\tilde{B}\subset \tilde{G}$ 
such that $\rho(B)\subset \tilde{B}$. We work with the set of dominant weights $\tilde{\Lambda}_{\geq 0}$, 
$\Lambda_{\geq 0}$ defined by this choice: the Borel subgroup $\tilde{B}$ fix an ordering $>$ on 
the elements of $J$, and a weight $\tilde{\xi}=\sum_{j\in J} a_j e_j$ belongs to $\tilde{\Lambda}_{\geq 0}$ only if 
$j>k\Longrightarrow a_j\geq a_k$. For simplicity we write $J=\{1,\ldots,N\}$ with the canonical ordering.

For the remaining part of this section we work with a fixed partition $\sigma=(\sigma_1,\sigma_2,\ldots,\sigma_N)$, 
and we denote $S_\sigma(V)$ the corresponding irreducible representation of $\U(V)$. We can represent $\sigma$ by the element $\sum_{j=1}^N \sigma_j e_j\in \tilde{\tgot}^*$ (that we still denote $\sigma$). Let $\mu=\pi(\sigma)=\sum_{j=1} \sigma_j\alpha_j \in \Lambda_{\geq 0}$.

Let $\{0=j_0>j_2>\cdots >j_p=N\}$ be the set of element $j\in [0,\ldots, N]$ such that $\sigma_{j+1}>\sigma_j$ 
or $j\in \{0,N\}$. We have an orthogonal decomposition $V=\oplus_{k=1}^{p} V_{[k]}$ where 
$V_{[k]}$ is the vector space generated by the $v_j$ for $j\in [j_{k-1}+1,\ldots,j_k]$. The nilradical 
$\tilde{\ngot}_\sigma$ of the parabolic subgroup $\tilde{\pgot}_\sigma\subset \glgot(V)$ 
corresponds to the set of endomorphisms $f$ such that $f(V_{[k]})\subset \oplus_{j<k} V_{[j]}$.

The following Lemma is proved in \cite{Montagard96}

\begin{lem}
Let $\ngot_\mu$ the nilradical of the parabolic subgroup $\pgot_\mu\subset\ggot_\C$. The morphism 
$d\rho:\ggot_\C\to \glgot(V)$ defines an injective map from $\ngot_\mu$ into $\tilde{\ngot}_\sigma$.
\end{lem}

We define $\mathbb{W}_\sigma$ as the quotient $\tilde{\ngot}_\sigma/\rho(\ngot_\mu)$. Recall that the 
image by $\rho$ of the stabiliser subgroup $G_\mu$ is contained in the stabilizer subgroup of $\sigma$: 
hence $\mathbb{W}_\sigma$ is a $G_\mu$-module.

For any partition 
$\theta=(\theta_1,\ldots,\theta_N)$, we associate the partition of length $\dim V_{[k]}$,  
$\theta_{[k]}:=(\theta_{j_{k-1}+1},\ldots, V_{j_k})$, and the irreducible representation 
$S_{\theta_{[k]}}(V_{[k]})$ of the unitary group $\U(V_{[k]})$.

For any partition $\theta$ of length $N$ and any dominant weight of $\lambda\in \Lambda_{\geq 0}$ we denote
$$
\left[V^G_{\lambda+n\mu}:S_{\theta+n\sigma}(V)\right]
$$
the multiplicity of the irreducible representation $V^G_{\lambda+n\mu}$ in the restriction 
$S_{\theta+n\sigma}(V)\vert_G$.

The following Theorem, which is a particular case of Theorem \ref{theoreme-B-second}, was first obtained 
by Manivel \cite{Manivel97} when $G=\U(E)$ and by Brion \cite{Brion93} when $\sigma=(1)$. 
The following version was obtained by Montagard \cite{Montagard96}: 
the only improvement that we obtain here is condition $a)$.

\begin{theo}
Let $\sigma$ a partition of length $\dim V$ and $\mu=\pi(\sigma)$.

The following statements are equivalent:
\begin{itemize}
\item[a)] $[V^G_{n\mu}:S_{n\sigma}(V)]=1$, for all $n\geq 1$.

\item[b)] For any couple $(\lambda,\theta)$ the increasing sequence 
$[V^G_{\lambda+n\mu}:S_{\theta+n\sigma}(V)]$ has a limit.

\item[c)] The algebra $\Sym(\mathbb{W_\sigma}^*)$ has finite $G_\mu$-multiplicities.
\end{itemize}

If these statements hold the limit of the sequence $[V^G_{\lambda+n\mu}:S_{\theta+n\sigma}(V)]$ is equal to 
the multiplicity of $V^{G_\mu}_{\lambda}$ in the $G_\mu$-module
$$
\Sym(\mathbb{W}^*_\sigma)\otimes S_{\theta_{[1]}}(V_{[1]})\otimes S_{\theta_{[2]}}(V_{[2]})\otimes \cdots \otimes 
S_{\theta_{[p]}}(V_{[p]}).
$$
\end{theo}

\bigskip



\end{document}